\documentclass[a4paper]{amsart}
\usepackage{a4wide}
\usepackage{amsmath}
\usepackage{amsthm}
\usepackage{amsfonts}
\usepackage{amssymb}
\usepackage{mathrsfs}
\usepackage{tikz}
\usepackage[alphabetic,initials,nobysame,non-compressed-cites]{amsrefs}
\usepackage{url}
\usepackage[hidelinks]{hyperref}
\usepackage{graphicx}
\usepackage{overpic}
\usepackage{todonotes}
\usepackage{enumitem}
\usepackage{array}
\usepackage{pdflscape}
\usepackage{lineno}

\setenumerate{label=\rm{(\roman*)}}
\theoremstyle{plain}
\newtheorem{theorem}{Theorem}
\newtheorem{proposition}{Proposition}[section]
\newtheorem{lemma}[proposition]{Lemma}
\newtheorem{question}[proposition]{Question}
\newtheorem{conjecture}[proposition]{Conjecture}
\newtheorem{corollary}[proposition]{Corollary}
\makeatletter
\newtheorem*{rep@theorem}{\rep@title}
\newcommand{\newreptheorem}[2]{%
\newenvironment{rep#1}[1]{%
 \def\rep@title{#2 \ref{##1}}%
 \begin{rep@theorem}}%
 {\end{rep@theorem}}}
\makeatother

\newreptheorem{theorem}{Theorem}
\newreptheorem{lemma}{Lemma}
\theoremstyle{remark}
\newtheorem{remark}[proposition]{Remark}

\theoremstyle{definition}
\newtheorem{definition}[proposition]{Definition}
\newtheorem{example}[proposition]{Example}

\def\Euc{\mathbb{E}}
\def\Sph{\mathbb{S}}
\def\Lor{\mathbb{L}}
\def\Hyp{\mathbb{H}}

\def\RR{\mathbb{R}}
\def\theBall{\mathcal{B}}
\def\theLightCone{\mathcal{L}}
\def\theSphere{\mathcal{S}}
\def\cC{\mathcal{C}}
\def\cE{\mathcal{E}}
\def\cK{\mathcal{K}}
\def\cP{\mathcal{P}}
\def\cQ{\mathcal{Q}}
\def\cT{\mathcal{T}}
\def\cV{\mathcal{V}}
\def \bx {\mathbf x}
\def \by {\mathbf y}

\DeclareMathOperator{\conv}{conv}
\DeclareMathOperator{\relint}{relint}
\DeclareMathOperator{\Span}{span}

\DeclareMathOperator{\cone}{cone}
\newcommand{\polar}[1]{{#1}^*} %
\newcommand{\ass}[1]{{#1}^\diamond} %
\DeclareMathOperator{\interior}{int}
\DeclareMathOperator{\homog}{hom} %
\newcommand{\set}[2]{\ensuremath{\left\{#1\,\middle|\,#2\right\}}} 
\newcommand{\floor}[1]{\left\lfloor {#1} \right\rfloor} %
\newcommand{\defn}[1]{\emph{\color{blue} #1}} %
\newcommand{\sprod}[2]{\langle {#1} , {#2} \rangle} %
\newcommand{\Lprod}[2]{\sprod{#1}{#2}} %
\newcommand{\Eprod}[2]{( {#1} , {#2} )} %

\newcommand{\cyc}[2]{\cC_{#1}({#2})} %

\title{Scribability problems for polytopes}
\author{Hao Chen}
\address{Freie Universit\"at Berlin, Institut f\"ur Mathematik, Arnimallee 2, 14195 Berlin, Deutschland}
\email{hao.chen.math@gmail.com}
\thanks{H.\ Chen was supported by the ERC Advanced Grant number 247029 ``SDModels''.}

\author{Arnau Padrol}
\address{Sorbonne Universit\'es, Universit\'e Pierre et Marie Curie (Paris 6), Institut de Math\'ematiques de Jussieu - Paris Rive Gauche (UMR 7586), Case 247, 4 place Jussieu, 75252 Paris Cedex 05, France}
\email{arnau.padrol@imj-prg.fr}
\thanks{A.\ Padrol thanks the support of the  DFG Collaborative Research Center SFB/TR~109 ``Discretization in Geometry and Dynamics'' as well as the program PEPS
Jeunes Chercheur-e-s 2016 of the INSMI (CNRS)}


\begin{document}

\begin{abstract}
	In this paper we study various scribability problems for polytopes.  We begin
	with the classical $k$-scribability problem proposed by Steiner and
	generalized by Schulte, which asks about the existence of $d$-polytopes that
	cannot be realized with all $k$-faces tangent to a sphere.  We answer this
	problem for stacked and cyclic polytopes for all values of $d$ and $k$.  We
	then continue with the weak scribability problem proposed by Gr\"unbaum and
	Shephard, for which we complete the work of Schulte by presenting
	non weakly circumscribable $3$-polytopes. Finally, we propose new
	$(i,j)$-scribability problems, in a strong and a weak version, which
	generalize the classical ones. They ask about the existence of $d$-polytopes
	that can not be realized with all their $i$-faces ``avoiding'' the sphere and
	all their $j$-faces ``cutting'' the sphere.  We provide such
	examples for all the cases where $j-i \le d-3$.
\end{abstract}

\maketitle

\tableofcontents

\section{Introduction}\label{sec:intro}
The history of scribability problems goes back to at least 1832, when Steiner
asked whether every $3$-polytope is inscribable or
circumscribable~\cite{steiner1832}.  A polytope is \defn{inscribable} if it can
be realized with all its vertices on a sphere, and \defn{circumscribable} if it
can be realized with all its facets tangent to a sphere. Steiner's problem
remained open for nearly 100 years, until Steinitz showed that inscribability
and circumscribability are dual through polarity, and presented a technique to
construct infinitely many non-circumscribable
$3$-polytopes~\cite{steinitz1928}.  A full characterization of inscribable
$3$-polytopes had to wait still more than 60 years, until Rivin gave one in
terms of hyperbolic dihedral angles~\cite{rivin1996} (see
also~\cite{hodgson1992, rivin1993, rivin1994, rivin2003}).  This was recently
expanded by Danciger, Maloni and Schlenker~\cite{danciger2014}, who obtained
Rivin-style characterizations for $3$-polytopes that are inscribable in a
cylinder or a hyperboloid.

A natural generalization in higher dimensions is to consider realizations of
$d$-polytopes with all their $k$-faces tangent to a sphere.  A polytope with
such a realization is said to be \defn{$k$-scribable}.  This concept was
studied by Schulte~\cite{schulte1987}, who constructed examples of
$d$-polytopes that are not $k$-scribable for all the cases except for $k=1$ and
$d=3$ and the trivial cases of $d\leq 2$.  In fact, every $3$-polytope has a
realization with all its edges tangent to a sphere.  This follows from
Koebe--Andreev--Thurston's remarkable Circle Packing Theorem~\cite{koebe1936,
andreev1971, andreev1971b, thurston1997} because edge-scribed $3$-polytopes are
strongly related to circle packings; see~\cite{ziegler2007} for a nice
exposition, and~\cite{chen2016} for a discussion in higher dimensions.  This
was later generalized by Schramm~\cite{schramm1992}, who showed that an
edge-tangent realization exists even if the sphere is replaced by an arbitrary
strictly convex body with smooth boundary.

Schulte~\cite{schulte1987} also proposed a weak version of $k$-scribability,
following an idea of Gr\"unbaum and Shephard~\cite{grunbaum1987}.  A
$d$-polytope is \defn{weakly $k$-scribable} if it can be realized with the
affine hulls of all its $k$-faces tangent to a sphere.  Schulte was able to
construct examples of $d$-polytopes that are not weakly $k$-scribable for all
$k < d-2$, and left open the cases $k=d-2$ and $k=d-1$.

Scribability problems expose the intricate interplay between combinatorial and
geometric properties of convex polytopes. They arise naturally from several
seemingly unrelated contexts.  Inscribed polytopes are in correspondence with
Delaunay subdivisions of point sets. This makes their combinatorial properties
of great interest in computational geometry. They can also be interpreted as
ideal polyhedra in the Klein model of the hyperbolic space.  Moreover, several
polytope constructions are based on the existence of $k$-scribed polytopes.
For example, edge-scribed $4$-polytopes are used by the \defn{$E$-construction}
to produce $2$-simple $2$-simplicial fat polytopes~\cite{eppstein2003}. The
\defn{$E_t$-construction}, a generalization in higher dimension, exploits
$t$-scribed polytopes~\cite{paffenholz2004}.

However, our understanding on scribability properties is still quite limited.
As Gr\"unbaum and Shephard put it~\cite{grunbaum1987}: ``it is surprising that
many simple and tangible questions concerning them remain unanswered''.

In this paper we study classical \defn{$k$-scribability} problems, as well as a
generalization, \defn{$(i,j)$-scribability}, in both their strong and weak
forms.  The latter is the main object of interest of this paper; many of our
results about $k$-scribability arise as consequences of our findings on
$(i,j)$-scribability.  We focus on the existence problems.  Hence, for a family
of polytopes, we either construct scribed realizations for each of them, or
find explicit instances that are not scribable.  As we have seen, scribability
problems lie in the confluence of polyhedral combinatorics, sphere
configurations, and hyperbolic geometry; and our proof techniques and
constructions draw from all these areas.

\subsection{$k$-scribability}

Our investigation on the strong $k$-scribability problem focuses on two
important families of polytopes: \defn{stacked polytopes} and \defn{cyclic (and
neighborly) polytopes}.

By Barnette's Lower Bound Theorem~\cite{barnette1971, barnette1973}, stacked
polytopes have the minimum number of faces among all simplicial polytopes with
the same number of vertices.  The \defn{triakis tetrahedron} is a stacked
polytope among the first and smallest examples of non-inscribable polytopes
found by Steinitz~\cite{steinitz1928}.  Recently, Gonska and
Ziegler~\cite{gonska2013} completely characterized inscribable stacked
polytopes.  On the other hand, Eppstein, Kuperberg and
Ziegler~\cite{eppstein2003} showed that stacked $4$-polytopes are essentially
not edge-scribable.

In Section~\ref{sec:stacked}, we look at the other side of the story and prove
the following result, which completely answers the $k$-scribability problems
for stacked polytopes.
\begin{theorem}\label{thm:stacked}
	For any $d \ge 3$ and $0 \le k \le d-3$, there are stacked $d$-polytopes that
	are not $k$-scribable.  However, every stacked $d$-polytope is
	$(d-1)$-scribable (i.e.\ circumscribable) and $(d-2)$-scribable (i.e.\
	ridge-scribable).
\end{theorem}
The proof of Theorem~\ref{thm:stacked} is divided into three parts:
Proposition~\ref{prop:face} for $0 \le k \le d-3$, Proposition~\ref{prop:ridge}
for $k=d-2$ and Proposition~\ref{prop:facet} for $k=d-1$.  The construction for
ridge-scribability extends in much more generality to connected sums of
polytopes. The circumscribability statement is equivalent to saying that
truncated polytopes are always inscribable.  In Proposition~\ref{prop:facet},
we show a stronger statement, that truncated polytopes are inscribable in any
strictly convex surface. Hence, truncated polytopes are new examples of
\defn{universally inscribable} polytopes in the sense of~\cite{gonskapadrol2016}.

On the other end of the spectrum, McMullen's Upper Bound
Theorem~\cite{mcmullen1970} states that cyclic polytopes have the maximum
number of faces among all polytopes with the same number of vertices. All
cyclic polytopes are inscribable; see~\cite{caratheodory1911},
\cite[p.~67]{grunbaum1987} \cite[p.~521]{seidel1991}
and~\cite[Proposition~17]{gonska2013} (and even universally
inscribable~\cite{gonskapadrol2016}).  The following theorem completely answers the
$k$-scribability problem for cyclic polytopes.
\begin{theorem}\label{thm:cyclic}
	For any $d > 3$ and $1\le k\le d-1$, a cyclic $d$-polytope with
	sufficiently many vertices is not $k$-scribable.
\end{theorem}
Theorem~\ref{thm:cyclic} is derived from more general results on
$(i,j)$-scribability problems, namely Corollary~\ref{cor:2d-1} and
Proposition~\ref{prop:neighborly}.  The latter concerns $k$-neighborly
polytopes, i.e.\ polytopes with a complete $k$-skeleton.  Cyclic polytopes are
$\lfloor d/2 \rfloor$-neighborly, and (simplicial) $\lfloor d/2
\rfloor$-neighborly polytopes (usually called simply neighborly) are precisely
those with the same number of faces as the corresponding cyclic polytope.
Proposition~\ref{prop:neighborly} implies that $j$-neighborly $d$-polytopes
with sufficiently many vertices are not $k$-scribable for $1 \le k \le j$, from
which we derive that:
\begin{theorem}\label{thm:fvector}
	For any $d > 3$ and $1 \le k \le d-2$, there are $f$-vectors such that no
	$d$-polytope with those $f$-vectors are $k$-scribable.
\end{theorem}
For $k=d-3$, examples of such $f$-vectors are already given by stacked
$d$-polytopes with more than $d+2$ vertices; see~\cite{eppstein2003} for $d=4$
and Proposition~\ref{prop:face} for higher dimensions.  It is still an open
question whether all $f$-vectors of polytopes appear as $f$-vectors of
inscribable polytopes~\cite{gonska2013}.  This could be settled by a solution
to Conjecture~\ref{conj:neighnocircum}, which asserts that neighborly polytopes
with sufficiently many vertices are not circumscribable .

We also look into \defn{weak $k$-scribability}.  The main difficulty that
prevented Schulte from settling the $k=d-2$ and $k=d-1$ cases
in~\cite{schulte1987} was that weak scribability does not behave well under
polarity. In this paper, we identify convex polytopes with pointed polyhedral
cones in Lorentzian space (see Section~\ref{sec:prelim} for details).  In this
set-up, the definition of weak scribability is slightly weaker than in
Euclidean space, but behaves well under polarity.  This allows us to construct
examples that prove the following theorem.  Note that, since our definition
imposes weaker conditions, these examples are also not weakly $k$-scribable in
Schulte's sense, hence closing the cases left open by Schulte.
\begin{theorem}\label{thm:weakk}
	For any $d\geq 3$ and $0\le k \le d-1$ with the exception of $(d,k)=(3,1)$, there
	are $d$-polytopes that are not weakly $k$-scribable.
\end{theorem}

\subsection{$(i,j)$-scribability}

We propose the new concept of \defn{$(i,j)$-scribability}.  A polytope is
\defn{$(i,j)$-scribable} if it can be realized with all its $i$-faces
``avoiding'' the sphere and all its $j$-faces ``cutting'' the sphere.  The
definitions of cutting and avoiding come with a \defn{strong} and a \defn{weak}
form; see~Definition~\ref{def:cutavoid}.  They are designed to behave well
under polarity, and to reduce to classical $k$-scribability when $i=j=k$.  This
makes $(i,j)$-scribability a very useful tool for studying classical
$k$-scribability problems, as we will see with cyclic polytopes.  They are also
an interesting topic in their own right.  Notably, $(0,1)$-scribed
$3$-polytopes have been studied as \defn{hyperideal polyhedra} in hyperbolic
space~\cite{bao2002, schlenker2005}.

One of our main results is the following theorem, which constructs examples of
polytopes that are not $(i,j)$-scribable.

\begin{theorem}\label{thm:strongij}
	For $d>3$ and $0\le i\le j\le d-1$, there are $d$-polytopes that are not
	strongly $(i,j)$-scribable for $j-i\le d-2$ when $d$ is even, or $j-i \le
	d-3$ when $d$ is odd.
\end{theorem}
Theorem~\ref{thm:strongij} follows from Proposition~\ref{prop:even}, which
asserts that even dimensional cyclic polytopes with sufficiently many vertices
are not strongly $(1,d-1)$-scribable.  The proof uses the Sphere Separator
Theorem~\cite{miller1997}.  Theorems~\ref{thm:cyclic} and \ref{thm:fvector}
mentioned before arise as corollaries.  An alternative construction for
polytopes that are not strongly $(0,d-3)$-scribable is given in
Proposition~\ref{prop:0d-3}, which uses the Stamp Theorem of
Below~\cite{below2002} and Dobbins~\cite{dobbins2011}. 

As for stacked polytopes, we can show that
\begin{theorem}\label{thm:stack01}
	For any $d>3$ and $0\leq i\leq d-4$, there are stacked $d$-polytopes that are
	not $(i,i+1)$-scribable.
\end{theorem}

The existence of $d$-polytopes that are not $(0,d-1)$-scribable is an
interesting open problem.  It essentially asks whether every polytope has a
realization with all its vertices outside a sphere and all its facets cutting
it.  This is a self-polar property. Note that every $3$-polytope is
$(0,2)$-scribable by the Circle Packing Theorem. Moreover, cyclic and stacked
$d$-polytopes, i.e.\ both ends of the simplicial $f$-vector spectrum, are also
$(0,d-1)$-scribable. We suspect nevertheless that there exist polytopes that
are not $(0,d-1)$-scribable, although our constructions cannot provide them.

In contrast, weak $(i,j)$-scribability is not so challenging. It turns out to
be indeed quite weak, as we can see in the following theorem
(cf.~Theorem~\ref{thm:weakk}).
\begin{theorem}\label{thm:weakij}
	Every $d$-polytope is weakly $(i,j)$-scribable for $0 \le i < j \le d-1$.
\end{theorem}

\medskip

The paper is organized as follows: Section~\ref{sec:prelim} is dedicated to
introducing the set-up for our scribability problems: polyhedral cones in Lorentz
space and spherical polytopes.  The different scribability problems that we
work with are presented in Section~\ref{sec:definition}.
Section~\ref{sec:stacked} contains our results about stacked and truncated
polytopes, and Section~\ref{sec:neighborly} contains those about cyclic and
neighborly polytopes.  An alternative technique for constructing polytopes that
are not $(i,j)$-scribable is given in Section~\ref{sec:stamp}. Finally, in
Section~\ref{sec:open}, we present some open problems and conjectures.

\subsection*{Acknowledgements}

We want to thank Karim A. Adiprasito, Ivan Izmestiev and G\"unter M. Ziegler
for sharing their useful suggestions and insights during many interesting
conversations on these topics.

\section{Lorentzian view of polytopes}\label{sec:prelim}
Classical scribability problems only consider bounded convex polytopes in
Euclidean space.  To define polarity properly in this setup, one must assume
that the polytope contains the origin in its interior.  This presents the major
difficulty in Schulte's work on weak $k$-scribability, and also leads to a
minor flaw in his proof regarding strong $k$-scribability (see
Remark~\ref{rmk:schulte}). 

We find it more natural and convenient to work with spherical polytopes, which
arise from pointed polyhedral cones in Lorentzian space.  This section
is dedicated to the introduction of this setup.  The main advantage is that,
for spherical polytopes, polarity is always well-defined and well-behaved.
This facilitates the study of weak scribability and enables us to obtain
Theorem~\ref{thm:weakk}.  At the same time, as we will see in
Lemma~\ref{lem:equivalent}, strong scribability in spherical space and in
Euclidean space are equivalent, so the new setting is compatible with previous
studies.  In fact, the presence of spherical geometry is necessary only in few
occasions (e.g.\ Example~\ref{ex:weak}).

\subsection{Convex polyhedral cones in Lorentzian space}

A (closed) non-empty subset of $\RR^{d+1}$ is a \defn{convex cone} if it is
closed under positive linear combinations.  A convex cone is \defn{pointed} if
it does not contain any linear subspace of $\RR^{d+1}$.  A convex cone is
\defn{polyhedral} if it is the conical hull of finitely many vectors
in~$\RR^{d+1}$, i.e.\ a set $\cK$ of the form 
\[
	\cK=\cone(V):=\set{\sum \lambda_i v_i}{\lambda_i\geq 0, v_i\in V}
\]
for some finite set $V\subset\RR^{d+1}$. 

Let $\cK$ be a convex cone and let $H\subset\RR^{d+1}$ be a linear hyperplane
disjoint from the interior of~$\cK$. We say that $H$ is \defn{supporting} for
$\cK$ if $H\cap\cK$ contains non-zero vectors. In this case, the \emph{closed}
half-space $H^-$ that contains~$\cK$ is called a \defn{supporting half-space}.

If $\cK$ is polyhedral, the intersections of $\cK$ with its supporting
hyperplanes are its \defn{faces}.  The set of faces ordered by inclusion forms
the \defn{face lattice} of $\cK$. A polyhedral cone $\cK'$ is
\defn{combinatorially equivalent} to $\cK$ if their face lattices are
isomorphic. In this case, we also say $\cK$ and $\cK'$ have the same
\defn{combinatorial type}, or that $\cK'$ is a \defn{realization} of $\cK$.  

\medskip

The \defn{Lorentzian space} $\Lor^{1,d}$ is $\RR^{d+1}$ endowed with the
Lorentzian scalar product:
\[
	\Lprod{\bx}{\by}:=-x_0y_0+x_1y_1+\dots+x_dy_d,\quad \bx,\by\in\RR^{d+1}.
\]
The \defn{light cone} of $\Lor^{1,d}$ is the pointed convex cone
\[
	\theLightCone := \set{\bx\in\Lor^{1,d}}{\Lprod{\bx}{\bx}\le 0,\ x_0\ge 0}.
\]
\begin{remark}
	In the literature, the term ``\emph{light cone}'' usually denotes the set of
	vectors $\bx$ such that $\Lprod{\bx}{\bx}=0$, i.e.\ the boundary of
	$\theLightCone \cup -\theLightCone$, which is not a convex cone. 
\end{remark}

We work with the polarity induced by the Lorentzian scalar product.  Hence, the
\defn{polar} of a set $X \subseteq \Lor^{1,d}$ is the convex cone
\[
	\polar{X} :=\set{ \bx \in \Lor^{1,d}}{\Lprod{\bx}{\by} \le 0 \text{ for all }
\by \in \cK}.
\]
We also define the \defn{orthogonal companion} $X^\perp$ as the polar of
its linear span
\[
	X^\perp := \polar{\Span(X)} = \set{ \bx }{ \langle \bx, \by \rangle = 0 \text{
	for all }\by \in X }.
\]
The light cone is self-polar, i.e.\ $\polar{\theLightCone}=\theLightCone$.  If
$\cV \subset \Lor^{1,d}$ is a linear subspace, then $\polar{\cV}=\cV^\perp$.

If $\cK$ is a pointed polyhedral cone, then $\polar{\cK}$ is
\defn{combinatorially dual} to~$\cK$. That is, the face lattice of
$\polar{\cK}$ is obtained from that of $\cK$ by reversing the inclusion
relations.  Indeed, to each face $F$ of $\cK$ is associated a face $\ass{F}$ of
$\polar{\cK}$, which we call the \defn{associated face} of $F$, given by
\[
	\ass{F} := F^\perp \cap \polar{\cK}.
\]

We recall some standard properties of polarity.
\begin{lemma}\label{lem:orthogonal}
	A subspace $\cV$ is disjoint from $\theLightCone$ if and only if $\cV^\perp$
	intersects the interior of $\theLightCone$, and $\cV$ is tangent to
	$\theLightCone$ (that is, $\cV \cap \theLightCone$ consists of a single ray) if and
	only if $\cV^\perp$ is tangent to $\theLightCone$.
\end{lemma}

\begin{lemma}\label{lem:polar}
  If $\bx$ is a non-zero vector on the boundary of pointed polyhedral cone
  $\cK$ and belongs to a face $F$, then $\bx^\perp$ is a supporting hyperplane
  for $\polar{\cK}$ that contains the associate face $\ass{F}$, and
  $\polar{\bx}$ is the corresponding supporting half-space.
\end{lemma}

Let $F$ be a $k$-dimensional face of a pointed polyhedral cone $\cK$.  The
\defn{face figure} of $F$, denoted by $\cK/F$, is the projection of $\cK$ onto
the quotient space $\cK/\Span{F}$.  The combinatorial type of $\cK/F$ is
induced by the faces of $\cK$ that contain $F$ as a proper face.  So $\cK/F$
is combinatorially dual to the associated face $\ass{F}$.

\subsection{Spherical, Euclidean and hyperbolic polytopes}\label{ssec:spheuc}

Let $\Sph^d$ be the $d$-dimensional spherical space, identified with the set
\[
	\Sph^d=\set{\bx \in \RR^{d+1}}{\|\bx\|_2=1}.
\]
A \defn{spherical polytope} in $\Sph^d$ is an intersection of finitely many
hemispheres that does not contain any antipodal points.  Every pointed
polyhedral cone $\cK \subset \RR^{d+1}$ corresponds to a \defn{spherical
polytope}~$\cP$ in $\Sph^d$ given by $\cP = \cK \cap \Sph^d$, and every
spherical polytope arises this way.  The face lattice of $\cP$ is inherited
from $\cK$, and the polar of $\cK$ induces the polar of $\cP$ given by
$\polar{\cP} := \polar{\cK} \cap \Sph^d$.  Note that a hyperplane in $\Sph^d$
is the intersection $H \cap \Sph^d$ where $H$ is a linear hyperplane in
$\RR^{d+1}$, and a half-space in $\Sph^d$ is a hemisphere.  The light cone
appears as a spherical cap $\theBall := \theLightCone \cap \Sph^d$, whose
boundary is a $(d-1)$-sphere ``at latitude $45^\circ$N'', which we denote by
$\theSphere := \partial\theBall$ to avoid confusion with the ambient $d$-sphere
$\Sph^d$.

Every pointed polyhedral cone $\cK \subset \RR^{d+1}$ admits a
\defn{transversal hyperplane}, i.e.\ an affine hyperplane~$H$ intersecting
every ray of $\cK$.  If we identify $H$ with the Euclidean space $\Euc^d$, then
the intersection $\cP = \cK \cap H$ is a bounded convex $d$-polytope in
$\Euc^d$.  Conversely, every $d$-dimensional Euclidean polytope can be
\defn{homogenized} to the $(d+1)$-dimensional pointed polyhedral cone
$\homog(\cP) := \set{(\lambda,\lambda\bx)}{\bx \in \cP,\lambda\geq 0}$.  Again,
the face lattice of $\cP$ is inherited from $\cK$, and the polar of $\cK$
induces the polar of $\cP$.  Hence, from the combinatorial point of view, there
is no difference between spherical polytopes and Euclidean polytopes.

We usually identify $\Euc^d$ with the fixed hyperplane $H_0=\{\bx \mid x_0 =
1\}$, so that the light cone~$\theLightCone$ appears as the standard unit ball
in $\Euc^d$.  We abuse the notation and denote the Euclidean unit ball by
$\theBall$ and its boundary by $\theSphere$, as its spherical counterparts.
Bounded Euclidean $d$-polytopes in~$\Euc^d$ correspond to spherical
$d$-polytopes contained in the hemisphere $x_0>0$ of $\Sph^d$ through central
(gnomonic) projection from the origin.  Furthermore, the polarity induced by
the Lorentzian scalar product coincides with the classical polarity in
Euclidean space.  That is:
\[
	\polar{\cP}:=\set{\bx}{\Eprod{\bx}{\by} \le 1 \text{ for all } \by \in P}
\]
where $\Eprod{\cdot}{\cdot}$ denotes the Euclidean inner product.  For an
affine subspace $H \subset \Euc^d$, we use $H^\perp$ to denote the \defn{polar
subspace} 
\[
	H^\perp := \{ \bx \in \Euc^d \mid \Eprod{\bx}{\by} = 1 \text{ for all } \by
\in H \},
\]
which corresponds to the orthogonal companion in $\Lor^{1,d}$.

The unit ball $\theBall$ in the Euclidean space $\Euc^d$ can also be seen as
the Klein model of the hyperbolic space $\Hyp^d$, then $\cP \cap \theBall$
is a hyperbolic polytope.  In the current paper, polytopes of interest have no
vertex in the interior of $\theBall$, so they are not compact, or even of
infinite volume in the hyperbolic space.  The hyperbolic view is very useful
for the study of stacked polytopes in Section~\ref{sec:stacked}.  In
particular, we will make use of hyperbolic reflection groups and hyperbolic
dihedral angles for our proofs.  Readers unfamiliar with hyperbolic polytopes
are referred to~\cite{vinberg1993} or~\cite{ratcliffe2006}.

\defn{Lorentz transformations} are those linear transformations of~$\RR^{d+1}$
that preserve $\theLightCone$.  Lorentz transformations of $\Lor^{1,d}$
correspond to projective transformations of $\Euc^d$ that preserve
$\theSphere$, or M\"obius transformations of $\Sph^d$.  If we regard the
interior of $\theBall$ as the hyperbolic space $\Hyp^d$, then the M\"obius
transformations correspond to isometries of $\Hyp^d$.

\begin{remark}
	The term ``\emph{Lorentz transformation}'' is usually used to denote those
	linear transformations that preserve the Lorentzian inner product.  What we
	call a Lorentz transformation here, i.e. those preserving $\theLightCone$,
	correspond to what is usually known as the \emph{orthochronous Lorentz
	transformations}.
\end{remark}

We use the notation \defn{$\Span(X)$} to denote the linear span in $\Lor^{1,d}$,
the spherical span in $\Sph^d$, and the affine span in $\Euc^d$, depending on the
context.  

\section{Definitions and properties}\label{sec:definition}
\subsection{Strong $k$-scribability}

The classical scribability problems go back to Steiner~\cite{steiner1832}, and
were studied in full generality by Schulte~\cite{schulte1987}. We will present
them in terms of spherical polytopes, but as we will see soon, this formulation
is equivalent to the classical setup.

Consider a polytope $\cP\subset\Sph^d$ and let $F$ be a face of $\cP$.  We say
that $F$ is \defn{tangent} to $\theSphere$ if $\relint(F) \cap \theSphere$
consists of a single point, which is called the \defn{tangency
point} of $F$ and denoted by $t_F$.

\begin{definition}\label{def:classical}
	A {spherical} polytope $\cP$ is \defn{(strongly) $k$-scribed} if every
	$k$-face of $\cP$ is tangent to~$\theSphere$. A combinatorial type
	of polytope is \defn{(strongly) $k$-scribable} if it has a (strongly) 
	$k$-scribed realization (as a spherical polytope).
\end{definition}

Here, the adjective \emph{strong} is used to distinguish from the weak
scribability, which will be defined later.  We often omit the adjective since
this type of scribability problem is of the earliest and greatest interest.  We
also say that a polytope $\cP$ is \defn{inscribable}, \defn{edge-scribable},
\defn{ridge-scribable} or \defn{circumscribable} if it is $0$-,  $1$-, $(d-2)$-
or $(d-1)$-scribable, respectively.

If $\cP\subset\Sph^d$ is contained in the upper hemisphere $x_0>0$, then it
corresponds to a bounded polytope in $\Euc^d$ (identified with $H_0$), and our
definition of ``\emph{$k$-scribed}'' coincides with that of
Schulte~\cite{schulte1987}.  However, since not every spherical polytope is in
$x_0>0$, it is not straightforward to see that every $k$-scribable spherical
polytope has a $k$-scribed realization in Euclidean space.

\begin{lemma}\label{lem:equivalent}
	If a $d$-polytope $\cP \subset \Sph^d$ is $k$-scribable, then $\cP$ admits
	a $k$-scribed realization in the Euclidean space $\Euc^d$ that is
	bounded and contains the center of $\theSphere$ in its interior.
\end{lemma}

\begin{remark}\label{rmk:schulte}
	Before showing that our definition is indeed equivalent to that of
	Schulte~\cite{schulte1987}, let us first point out a related problem in
	Schulte's paper.  

	In~\cite[p.~507f.]{schulte1987}, Schulte uses a footnote
	of~\cite[p.~285]{grunbaum2003} which says that, for any point $x$ in the
	interior of $\theBall$, there is a projective transformation $T_x$ that
	preserves $\theSphere$ and sends $x$ to the center of $\theSphere$. Schulte
	argued that, if a polytope $\cP$ does not contain the center of $\theSphere$,
	one can send a point $x\in \cP\cap\theBall$ to the center of $\theSphere$
	using $T_x$, and then $T_x\cP$ is a polytope containing the center.

	\begin{figure}[htpb]
		\includegraphics[width=.75\textwidth]{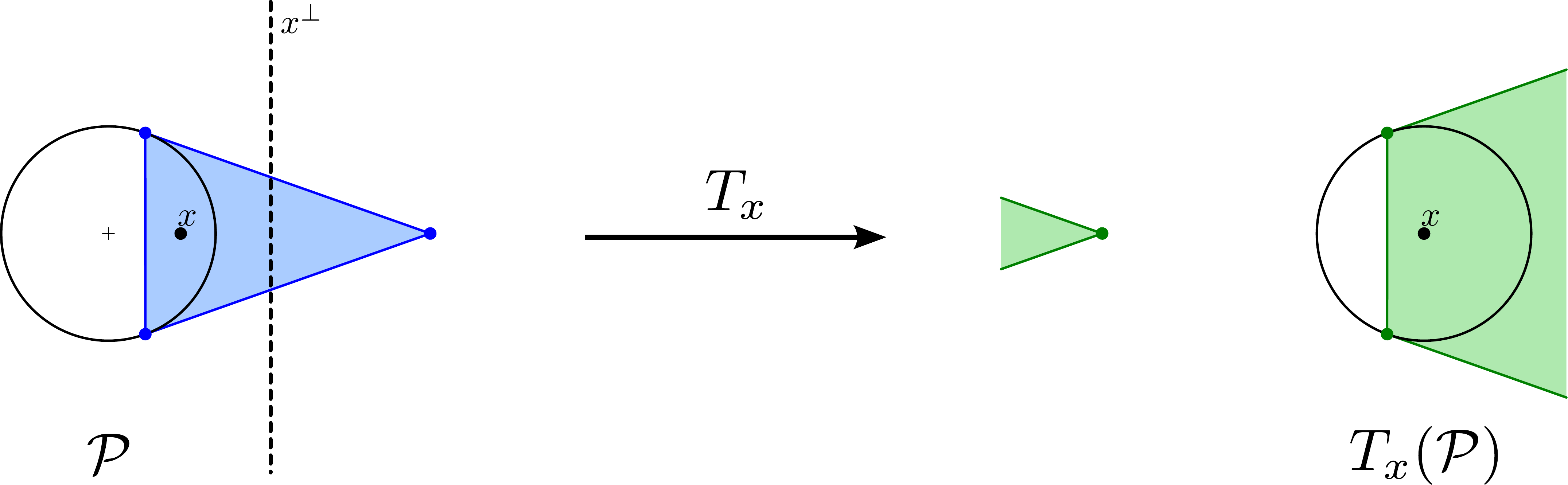}
		\caption{
			$T_x$ destroys the boundedness of the polytope.\label{fig:Schulte}
		}
	\end{figure}

	While this argument is correctly used in~\cite{grunbaum2003}, Schulte did not
	take precautions for the fact that, if the point $x$ is not carefully chosen,
	$T_x\cP$ might be unbounded (or even not connected).  For example, consider
	the triangle $\cP$ in Figure~\ref{fig:Schulte}, constructed by taking the
	convex hull of a point together with the intersections of its polar line with
	the circle.  For any point $x \in \cP \cap \theBall$, the \defn{polar
	hyperplane} $x^\perp$ is a line intersecting the triangle.  As $T_x$ sends
	$x$ to the center, $x^\perp$ is sent to infinity, which destroys the
	boundedness.  

	A correct argument in Euclidean space involves a careful choice of $x$;
	see~\cite[Lemma~6]{eppstein2003}.  Our proof of Lemma~\ref{lem:equivalent}
	uses the same idea, but more caution is required.
\end{remark}

\begin{proof}
	Assume that $\cP$ is $k$-scribed. 
	Now, for each $k$-face $F$ of $\cP$, we denote by $t_F \in \cP$ its tangency
	point. Observe that $F \subset H_F := t_F^\perp$ and $\theBall \subset H_F^-
	:= \polar{t_F}$.  We claim that $H_F^-$ is a supporting half-space for $\cP$.
	Indeed, for any $(k+1)$-face~$G \supset F$, we have $H_F \cap G = F$, but
	$H_F \cap \relint G = \emptyset$ because $G$ contains some interior point of
	$\theBall$ (the barycenter of the tangency points of its $k$-faces). Hence
	$H_F^-$ is supporting for $G$.  Since all the $k$-faces containing $F$ lie at
	the same side of $H_F$, we conclude that $H_F^-$ is supporting for all faces
	that contain $F$, including $\cP$ itself. 
	
	With a projective transformation if necessary, we may assume that $F$ is
	bounded.  Consider the subspace $L_F^\perp$ spanned by $F^\perp$ and the
	center of $\theBall$. It cuts $F$ transversally at $t_F$.  In particular,
	$\cP' = L_F^\perp \cap \cP$ is a $(d-k)$-polytope, $\theBall' = L_F^\perp
	\cap \theBall$ a $(d-k)$-ball, and $t_F$ is a vertex of $\cP'$ tangent
	to~$\theSphere' = \partial \theBall'$.  The half-space $L_F^\perp \cap H_F^-$
	is supporting both $\cP'$ and $\theBall'$.  This allows us to choose a point
	$x \in \interior (\cP' \cap \theBall')$.  Since $L_F^\perp$ is spanned by $x$
	and $F^\perp$, the intersection of $\Span(F)=F^{\perp\perp}$ and $x^\perp$
	coincides with the polar of $L_F^\perp$, which is contained in the equator at
	infinity because the center of $\theBall$ belongs to $L_F^\perp$. Hence,
	$\Span(F)$ and $x^\perp$ are parallel (intersect only at the equator at
	infinity), and $x^\perp$ is disjoint from $F$ as $F$ is bounded.	

	In particular, if $x$ is sufficiently close to $t_F$, then $x^\perp$ is
	disjoint from both $\cP$ and $\theBall$.  Thus the projective transformation
	$T_x$ will send $\cP$ to a bounded $k$-scribed polytope containing the center
	of $\theSphere$.
\end{proof}

As a consequence, $k$-scribability for Euclidean polytopes is equivalent to
$k$-scribability for spherical polytopes.

\subsection{Weak $k$-scribability}

Weak $k$-scribability problems were first asked for $3$-polytopes by Gr\"unbaum
and Shephard~\cite{grunbaum1987}, and then generalized to higher dimensions by
Schulte~\cite[Section~3]{schulte1987}.  By considering spherical polytopes,
the notion of weak scribability is weakened, but has the desired properties with
respect to polarity. 

\medskip

Consider a polytope $\cP\subset\Sph^d$, and let $F$ be a face of $\cP$.  We say
that $F$ is \defn{weakly tangent} to $\theSphere$ if its spherical span is
tangent to $\theSphere$, i.e. if $\Span(F) \cap \theSphere$ consists of a
single point.  In particular, since the spherical span of a point consists of itself 
together with its antipodal point in $\Sph^d$, a vertex is weakly tangent to~$\theSphere$ if
and only if it lies either on $\theSphere$ or on $-\theSphere$.

\begin{definition}\label{def:weak}
	A spherical polytope $\cP$ is \defn{weakly $k$-scribed} if every $k$-face of
	$\cP$ is weakly tangent to $\theSphere$. A combinatorial type of polytope
	is \defn{weakly $k$-scribable} if it has a $k$-scribed realization (as a spherical polytope).
\end{definition}

Definition~\ref{def:weak} is weaker than the Euclidean definition of
Gr\"unbaum--Shephard and Schulte.  The two definitions of strong
$k$-scribability were equivalent thanks to Lemma~\ref{lem:equivalent} which
shows that, for every strongly $k$-scribed polytope~$\cP$, there is a
hemisphere (e.g. $H_F^-$) containing both~$\cP$ and~$\theSphere$.  This is
however not true for weakly $k$-scribed polytopes; the spherical version is
weaker.  

For example, with the Gr\"unbaum--Shephard definition, any weakly inscribed
polytope is also strongly inscribed; see Schulte~\cite{schulte1987}.  This is
not the case with Definition~\ref{def:weak}.  The first example is, again,
given by a stacked polytope. The \defn{triakis tetrahedron}, which is the
polytope obtained by stacking a vertex on top of every facet of a
$3$-dimensional simplex, is a well-known polytope that is not strongly
inscribable~\cite{steinitz1928} (cf.\ \cite[Theorem 13.5.3]{grunbaum2003},
\cite{gonska2013}).  However, it is weakly inscribable.  

\begin{example}\label{ex:weak}
	The triakis tetrahedron is weakly inscribable (in the sense of
	Definition~\ref{def:weak}).
\end{example}

\begin{proof}
	Consider the following eight vectors in $\RR^4$.
	\begin{align*}
  	v_{1,2}&:=(+\sqrt{2},0,\pm 1,1)  & v_{3,4}&:=(+\sqrt{3},\pm \sqrt{2},0,1)\\
  	v_{5,6}&:=(-\sqrt{2},\pm 1,0,1)  & v_{7,8}&:=(-\sqrt{3},0,\pm \sqrt{2},1).
	\end{align*}
	They all satisfy $-x_0^2+x_1^2+x_2^2+x_3^2=0$.  Hence for all $1\le i \le 8$,
	the intersection of the line $\Span(v_i)$ and the light cone $\theLightCone
	\subset \RR^4$ is a single ray.  This means that the intersection of
	$\cK=\cone(v_1,\dots,v_8)$ with $\Sph^3$ is a weakly inscribed $3$-polytope
	$\cP$. Its combinatorial type is that of the triakis tetrahedron.
\end{proof}

Figure~\ref{fig:tetrahedra} illustrates the inscribed configuration.  On the
left, we visualize the intersection of $\cK$ with the hyperplane $x_3=1$;
notice that then the intersection of $\theLightCone \cup -\theLightCone$ with
the hyperplane appears as a $2$-sheet hyperboloid.  On the right, we visualize
it by considering the intersection of $\cK\cup -\cK$ with the hyperplane $x_0=1$;
then we see two connected components, corresponding to $\cK$ and $-\cK$,
respectively.

\begin{figure}[hptb]
	\includegraphics[width=.3\textwidth]{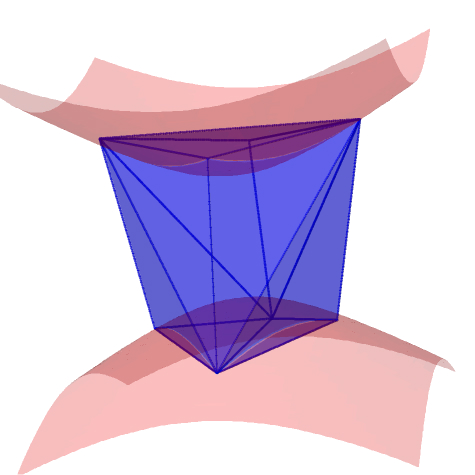}\qquad\qquad\qquad
	\includegraphics[width=.3\textwidth]{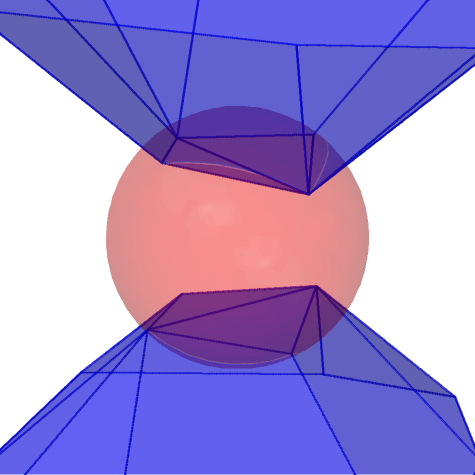}
	\caption{
		The inscribed triakis tetrahedron~$\cP$, view of $\pm\cP$ and $\pm
		\theSphere$ projected on the hyperplane $x_3=1$ (left), and on $x_0=1$
		(right).
	}\label{fig:tetrahedra}
\end{figure}

Nevertheless, Definition~\ref{def:weak} has the desired property: the polar of
a weakly $k$-scribable $d$-polytope is weakly $(d-1-k)$-scribable; see the upcoming
Lemma~\ref{lem:property}\ref{it:polar}.  This is precisely the missing piece
that prevented Schulte from proving Theorem~\ref{thm:weakk}.

\subsection{Strong and weak $(i,j)$-scribability}

In this section, we present the new concept of $(i,j)$-scribability, which
generalizes the concept of $k$-scribability presented above.  Instead of asking
for realizations with the $k$-faces tangent to the sphere, we are interested in
realizations with the $i$-faces ``avoiding'' the sphere and the $j$-faces
``cutting'' it, in such a way that $(i,j)$-scribability reduces to
$i$-scribability when $i=j$.  As before, the definitions come in strong and
weak forms. 

\begin{definition}\label{def:cutavoid}
	Consider a spherical polytope $\cP\subset\Sph^d$, and let $F$ be a proper
	face of~$\cP$.

	We say that~$F$
	\begin{itemize}
		\item \defn{strongly cuts} $\theSphere$ if $\relint(F) \cap \theBall \ne
			\emptyset$;

		\item \defn{weakly cuts} $\theSphere$ if $\Span(F) \cap \theBall \ne
			\emptyset$.

	\end{itemize}

	We say that~$F$
	\begin{itemize}
		\item \defn{strongly avoids} $\theSphere$ if there is a supporting hyperplane $H$ of $\cP$ such that $F = H\cap \cP$ and $\theBall \subset H^-$.

		\item \defn{weakly avoids} $\theSphere$ if $\Span(F) \cap \interior \theBall
			= \emptyset$;

	\end{itemize}
	where $\interior \theBall = \theBall \setminus \theSphere$ is the interior of
	the ball $\theBall$.
\end{definition}

\begin{example}\label{ex:ijstrong}
	Consider the following triangle.
	\begin{center}
 		\includegraphics[width=.4\linewidth]{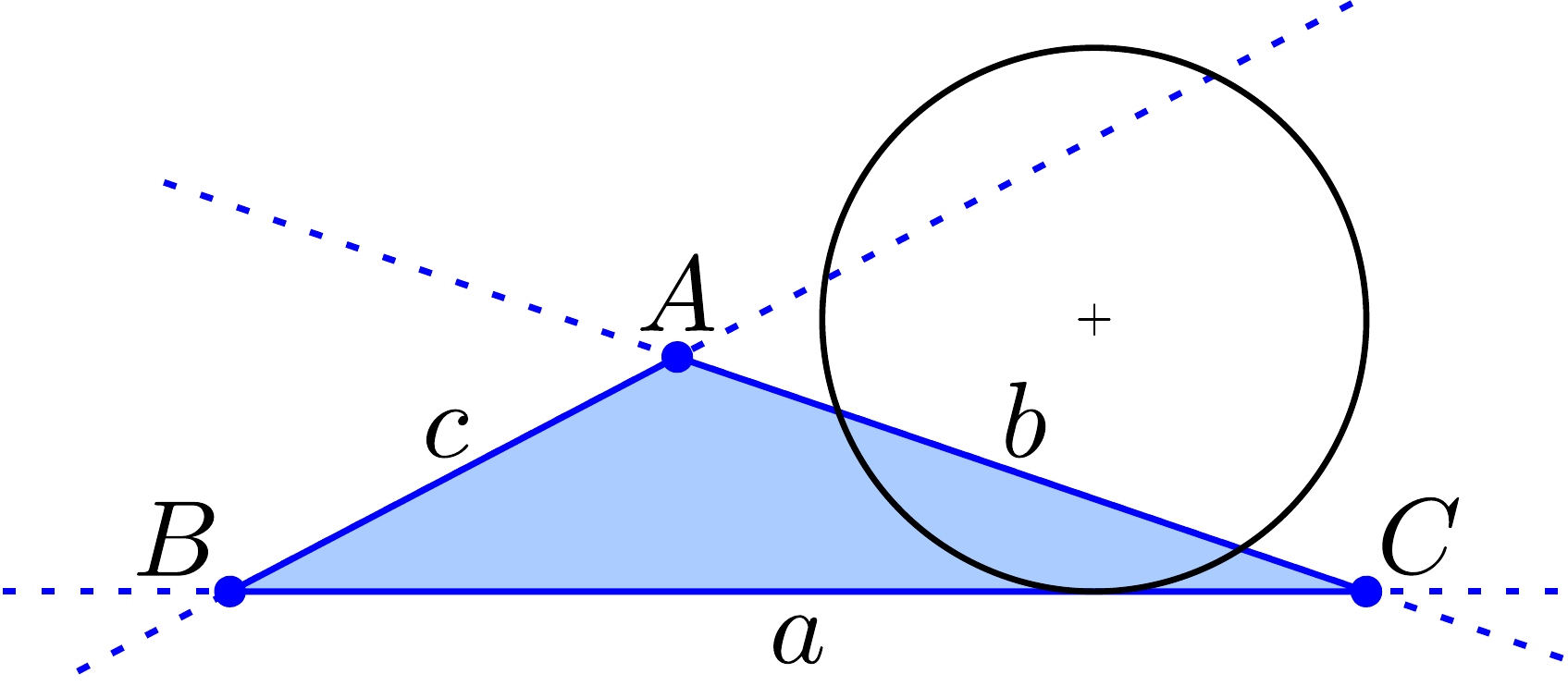}
	\end{center}
	The edge $a$ strongly avoids and cuts $\theSphere$.  The edge $b$ strongly
	cuts $\theSphere$ and does not avoid $\theSphere$ in any sense.  The edge $c$
	weakly cuts $\theSphere$ as shown by the dashed line.  The vertex $A$ weakly
	avoids $\theSphere$.  The vertices $B$ and $C$ strongly avoid $\theSphere$
	and do not cut $\theSphere$ in any sense.
\end{example}

In the following, the phrase ``in the strong (resp.\ weak) sense'' means that
the adverb ``strongly'' (resp.\ ``weakly'') is implied wherever applicable in
the context.

\begin{definition}
	Let $0 \le i \le j \le d-1$.  In the \defn{strong} or \defn{weak} sense, a
	spherical $d$-polytope $\cP \subset \Sph^d$ is \defn{$i$-avoiding} if every
	$i$-face of $\cP$ avoids $\theSphere$, and \defn{$j$-cutting} if every
	$j$-face of $\cP$ cuts $\theSphere$.  We say that $\cP$ is
	\defn{$(i,j)$-scribed} if it is $i$-avoiding and $j$-cutting. A combinatorial
	type of polytope is \defn{$(i,j)$-scribable} if it has an $(i,j)$-scribed
	realization (as a spherical polytope).
	
\end{definition}

\begin{remark}
	These notions should not be confused with a $(m,d)$-scribable polytope in the
	sense of Schulte~\cite{schulte1987}, which means a $d$-polytope that is
	$m$-scribable.
\end{remark}

\begin{remark}
 	There is a third  even weaker  version of scribability that seems reasonable
 	at first sight: faces are \defn{feebly cutting} if they are not strongly
 	avoiding and \defn{feebly avoiding} if they are not strongly cutting.
 	However, observe that we can always find such a feebly $(d-1,0)$-scribed
 	realization of any polytope $P$ by simply taking a realization of $P$ with
 	all the vertices in $-\theBall\subset\Sph^d$, and thus this feeble version of
 	scribability is trivial.

\end{remark}

\subsection{Properties of $(i,j)$-scribability}

A first easy observation is that the strong versions are indeed stronger than
the weak versions. That is, the strong forms of cutting, avoiding and
scribibability imply the weak forms. A second easy observation is that they
contain the classical scribability problems as a special case.

\begin{lemma}\label{lem:reduce}\leavevmode
	In the strong or weak sense, a face that cuts and avoids $\theSphere$ is
	tangent to $\theSphere$.  Consequently, a $(k,k)$-scribed polytope is
	$k$-scribed.
\end{lemma}

We say that a face is \defn{strictly} cutting (resp.\ avoiding) $\theSphere$ if
it is cutting (resp.\ avoiding) $\theSphere$ but not tangent to $\theSphere$.
In the strong sense, it is possible that a face is neither cutting nor avoiding
$\theSphere$ at the same time.  This is however not possible in the weak sense.

The following lemma collects the essential properties of $(i,j)$-scribability,
which are repeatedly used in the upcoming proofs.  It is a generalized version
of the main results in~\cite{schulte1987}, Theorems~1 and~2.  It implies also
the weak version of \cite[Theorem~1]{schulte1987}, needed for proving Theorem~\ref{thm:weakk}.

\begin{lemma}\label{lem:property}
	Let $d\ge 1$ and $0\le i,j\le d-1$.  In the strong or weak sense, if a
	$d$-polytope $\cP$ is $(i,j)$-scribable, then:
	\begin{enumerate}
		\item\label{it:otherij} $\cP$ is $(i',j')$-scribable for any $i'\leq i$ and
			$j'\geq j$;

		\item\label{it:polar} the polar polytope $\polar{\cP}$ is
			$(d-1-j,d-1-i)$-scribable;

		\item\label{it:facet} if $j\le d-2$, each facet of $\cP$ is an
			$(i,j)$-scribable $(d-1)$-polytope;

		\item\label{it:vf} if $i\ge 1$, each vertex-figure of $\cP$ is an
			$(i-1,j-1)$-scribable $(d-1)$-polytope.
	\end{enumerate}
\end{lemma}

\begin{proof}\leavevmode

	\begin{enumerate}
		\item Assume that $\cP$ is strongly $j$-cutting, we need to prove that
			$\cP$ is also $j'$-cutting for any $j' \ge j$.  Indeed, for any
			$(j+1)$-face $F$, take a point in $\relint F' \cap \theBall$ for each
			$j$-face $F'$ incident to $F$, then the barycenter of these points is in
			$\relint F \cap \theBall$.  By polarity we see that if $\cP$ is strongly
			$i$-avoiding, then $\cP$ is also $i'$-avoiding for any $i' \le i$.  This
			proves the strong version of the statement.
			The weak version follows similarly, just replace $\relint$ by $\Span$.  

		\item The weak version follows directly from Lemma~\ref{lem:orthogonal}.
			For the strong version, one also needs Lemma~\ref{lem:polar}.

		\item For either strong or weak version, notice from the proof of (i) that
			$\cP$ is strictly $(d-1)$-cutting in the weak sense.  Let $F$ be a facet.
			If we identify $\Span F \subset \Sph^d$ with $\Sph^{d-1}$, then $F\subset
			\Sph^{d-1}$ is an $(i,j)$-scribed $(d-1)$-polytope with respect to the
			sphere $\Span F \cap \theSphere$.

		\item By polarizing (iii).
			\qedhere
	\end{enumerate}
\end{proof}

In the remaining part of the paper, our goal is as follows: for each triple
$(d,i,j)$ we either try to prove that every $d$-polytope is $(i,j)$-scribable,
or we try to construct an example of a $d$-polytope that is not
$(i,j)$-scribable.  In view of Lemma~\ref{lem:property}(i), we seek to
construct polytopes that are not $(i,j)$-scribable with $j-i$ as large as
possible, or to prove that every $d$-polytope is $(i,j)$-scribable for $j-i$ as
small as possible.  The remaining items of Lemma~\ref{lem:property} are used to
do induction on the dimension $d$.

We end this section by proving that the Euclidean and spherical definitions of
strong $(i,j)$-scribability are equivalent.  When $i<j$, the analogue of
Lemma~\ref{lem:equivalent} does not guarantee simultaneously that $\cP$ is
bounded and contains the origin.  However, boundedness suffices for the
equivalence between the Euclidean and the spherical setup. 

\begin{lemma}\label{lem:newequivalent}
	If a $d$-polytope $\cP$ is strongly $(i,j)$-scribable, with $0\leq i\leq
	j\leq d-1$ then $\cP$ admits a strongly $(i,j)$-scribed bounded realization
	in Euclidean space $\Euc^d$.
\end{lemma}

\begin{proof}
	We prove the dual statement, i.e.\ $\cP$ admits a strongly $(i,j)$-scribed
	realization with the center of $\theSphere$ in its interior.  Since $\cP$ is
	strongly $(i,j)$-scribed, each facet contains a point of $\theBall$. The
	barycenter of these points is interior to both $\cP$ and $\theBall$, and can
	be sent to the center of $\theSphere$ with a projective transformation that
	fixes $\theBall$.  This gives a strongly $(i,j)$-scribed realization of $\cP$
	containing the center, then the polarity yields a bounded realization of
	$\cP^*$.
\end{proof}

\section{Weak scribability}\label{sec:weak}
In this section we concentrate on the investigation of weak
$(i,j)$-scribability.  On the one hand, we show in  Theorem~\ref{thm:weakk}
that there are $d$-polytopes that are not weakly $k$-scribable except for $d=3$
and $k=1$ or $d\leq 2$, since a $3$-polytope is already strongly
$1$-scribable. On the other hand, we show in Theorem~\ref{thm:weakij} that when
$i<j$, every polytope is weakly $(i,j)$-scribable.

\subsection{Weak $k$-scribability} \label{ssec:weak}

Despite Example~\ref{ex:weak}, there are $3$-polytopes that are not weakly
inscribable. We provide two constructions.

\begin{figure}[hptb]
	\includegraphics[width=.18\linewidth]{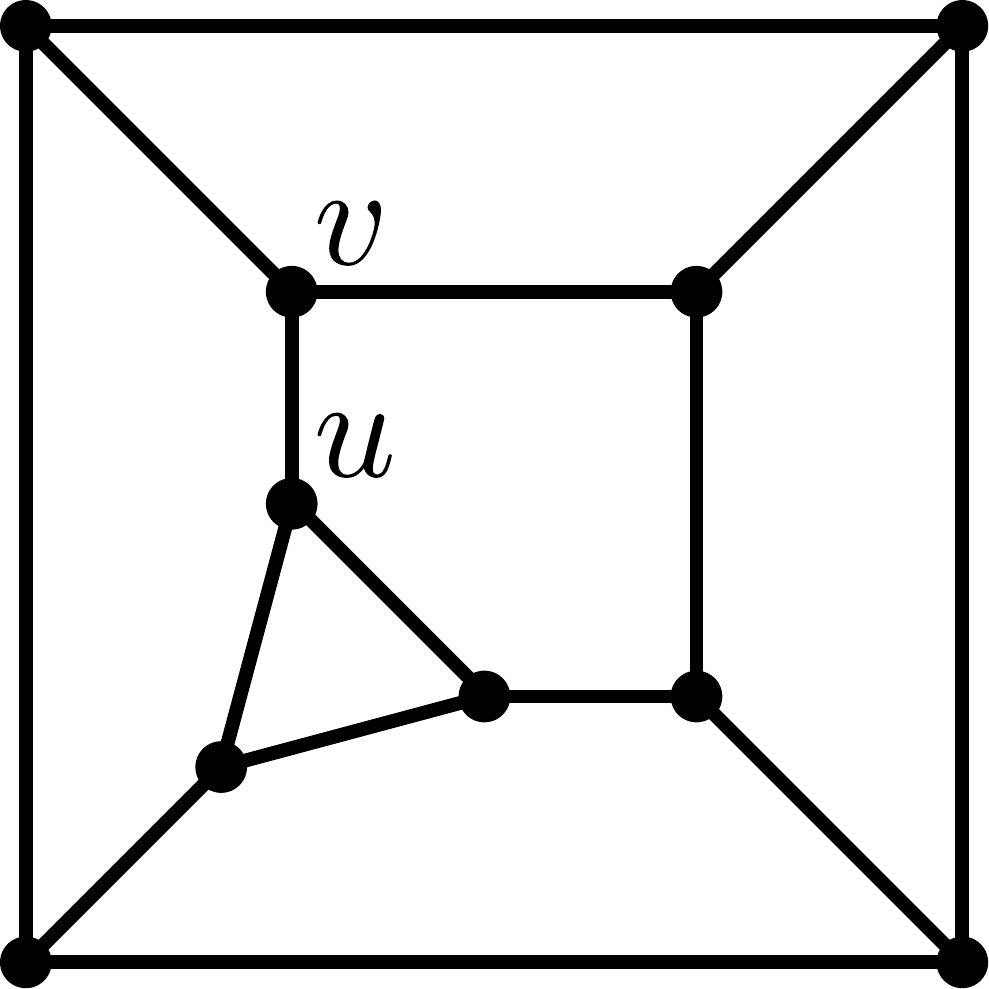}
	\caption{The truncated cube is not weakly inscribable.}\label{fig:trunc_cube}
\end{figure}

\begin{example}\label{exp:trunc_cube}
	The $3$-polytope $\cP$ obtained by truncating one vertex of a $3$-cube is not
	weakly inscribable.
\end{example}

\begin{proof}
	Let $\cK$ be the polyhedral cone spanned by a spherical realization of $\cP$
	and let $H$ be a transversal hyperplane of $\cK$.  If we identify the
	Euclidean space $\Euc^d$ with $H$ (instead of with $H_0$ as usual), the
	polytope $\cP' = H \cap \cK$ is a bounded polytope, and the ``sphere''
	$\theSphere = H \cap (\partial\theLightCone \cup -\partial\theLightCone)$
	appears as a quadric in $\Euc^d$.  Since $\cP$ is weakly inscribed, the
	vertices of $\cP'$ are all on the quadric~$\theSphere$.

	It is well known that if seven vertices of a (combinatorial) $3$-cube lie on
	a quadric, so does the eighth one \cite[Section~3.2]{bobenko2008}.  We can
	recover a $3$-cube by removing the truncating facet of $\cP'$.  Let $w$ be
	the 8th vertex of this cube (the one that does not belong to $\cP'$).  Then
	$w$ and the points $u$ and $v$ from Figure~\ref{fig:trunc_cube} all lie on
	the quadric.  If a quadric contains 3 collinear points, then it contains a
	whole line (see \cite[Ex~3.7]{bobenko2008}).  However, since $0\notin H$, our
	conic section does not contain lines; a contradiction.
\end{proof}

We are now ready to prove Theorem~\ref{thm:weakk}, which we repeat below.

\begin{reptheorem}{thm:weakk}
	Except for the case $d=3, k=1$, for every $d\geq 3$ and $0\leq k \leq d-1$, there
	is a $d$-polytope that is not weakly $k$-scribable.
\end{reptheorem}

Schulte~\cite[Sec.~3]{schulte1987} gave the proof for $k\leq d-3$, but remarked
that the remaining cases would follow from the existence of polytopes that are
not weakly circumscribable. The polar of the truncated cube is such a polytope.

\begin{proof}
	The proof is by induction on $d$ and follows the same steps as Schulte's.

	In dimension $3$, the truncated cube is not weakly inscribable, and its polar
	is not circumscribable by~\ref{lem:property}(i) (the weak version with $i = j
	= k$).  The pyramid over the truncated cube is a $4$-polytope that has a
	truncated cube as a facet and as a vertex figure, so it is not weakly
	$k$-scribable for $k=0,1$ by Lemma~\ref{lem:property}(ii) and (iii) (the weak
	version with $i = j = k$).  Its polar has a stacked octahedron as a facet and
	as a vertex figure, so it is not weakly $k$-scribable for $k=2,3$.

	In higher dimensions, the pyramid over a $(d-1)$-polytope that is not weakly
	$k$-scribable gives a $d$-polytope that is neither $k$- nor $(k+1)$-scribable.  So
	the theorem follows by induction.
\end{proof}

Recall that although we work in spherical space, our definition of weak
$k$-scribability is weaker than the one for the Euclidean setting used by Schulte.
Consequently, examples for Theorem~\ref{thm:weakk} are not weakly $k$-scribable
in Euclidean space, neither.  This finishes Schulte's work.

We now present an alternative construction.  Despite the absence of
Lemma~\ref{lem:property}(ii) in Euclidean space, it is possible to bypass the
spherical geometry, and construct polytopes that are not weakly circumscribable
directly within the Euclidean setup of~\cite{schulte1987}.  Our construction is
based on the following lemma:

\begin{lemma}\label{lem:noweaklypentagon}
	Any weakly circumscribed Euclidean polygon with more than $4$ edges is also
	strongly circumscribed.
\end{lemma}
\begin{proof}
Let $\cP$ be a weakly circumscribed polygon in Euclidean space.  A half-plane
$L^-$ supporting an edge of $\cP$ is either of the form $\{\sprod{a}{x} \le
-1\}$ for some unit vector $a$, in which case we call the edge
\defn{separating} because $\cP$ and $\theBall$ lie at opposite sides of the
supporting line, or of the form $\{\sprod{a}{x} \le 1\}$, and then
$\theBall\cup \cP\subset L^-$  and we call the edge \defn{non-separating}.
Observe that $\cP$ is not strongly circumscribed if and only if it has a
separating edge.  

Assume that $\cP$ has three separating edges.  If the unit normal vectors
were positively dependent, the intersection of the half-planes would be
empty. Hence, we may assume that one of them can be written as a linear
combination of the other two with positive coefficients, but then one can
easily see that this inequality is redundant.

We also claim that, with the presence of a separating edge, $\cP$ can have at
most two non-separating edges. To see this, consider one separating edge with
unit vector $a_0$ and two non-separating edges with vectors $a_1, a_2$.  A
linear relation of the form $a_1 = \lambda a_0 + \mu a_2$ with positive
coefficients is not possible.  Otherwise, a small computation shows that the
inequality defined by $a_1$ is redundant. 	The forbidden linear relation is
however inevitable if $\cP$ has three non-separating edges.

Hence, if $\cP$ has a separating edge, then it can have at most four edges,
two of each kind.  The different situations are illustrated in
Figure~\ref{fig:weaklycircumscribed}.  
\end{proof}

\begin{figure}[hptb]
	\includegraphics[width=.6\linewidth]{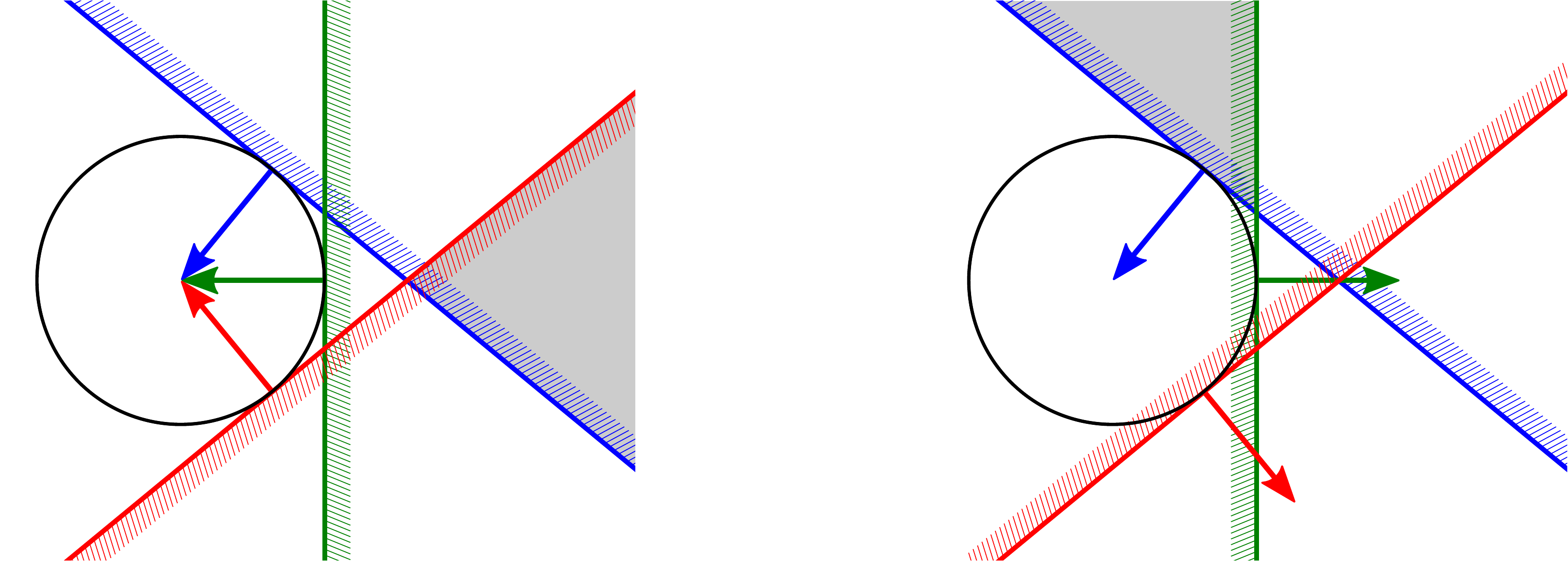}
	\caption{
		Redundant inequalities for the proof of Lemma~\ref{lem:noweaklypentagon}.
		The depicted vectors are outer normal vectors and point away from the
		half-space. The shaded region is the intersection of all the half-spaces.
		Notice that in both situations there is a half-space containing the
		intersection of the other two.
	}\label{fig:weaklycircumscribed}
\end{figure}

The same proof carries over almost directly to spherical polytopes. The correct
statement in the spherical setting would be: If $\cP\subset\Sph^2$ is a weakly
circumscribed spherical polygon with more than $4$ edges, then either $\cP$ or
$-\cP$ is strongly circumscribed.

\begin{example}\label{ex:2stdodecahedron}
 	Let $\cP$ be the polytope obtained by truncating all vertices of a
 	tetrahedron, then stacking on each of the newly created facets, and then
 	stacking again on each of the newly created facets
 	(Figure~\ref{fig:2stsimplex}).  Then $\cP$ is not weakly circumscribable. 
\end{example}

\begin{figure}[hptb]
	\includegraphics[width=.2\linewidth]{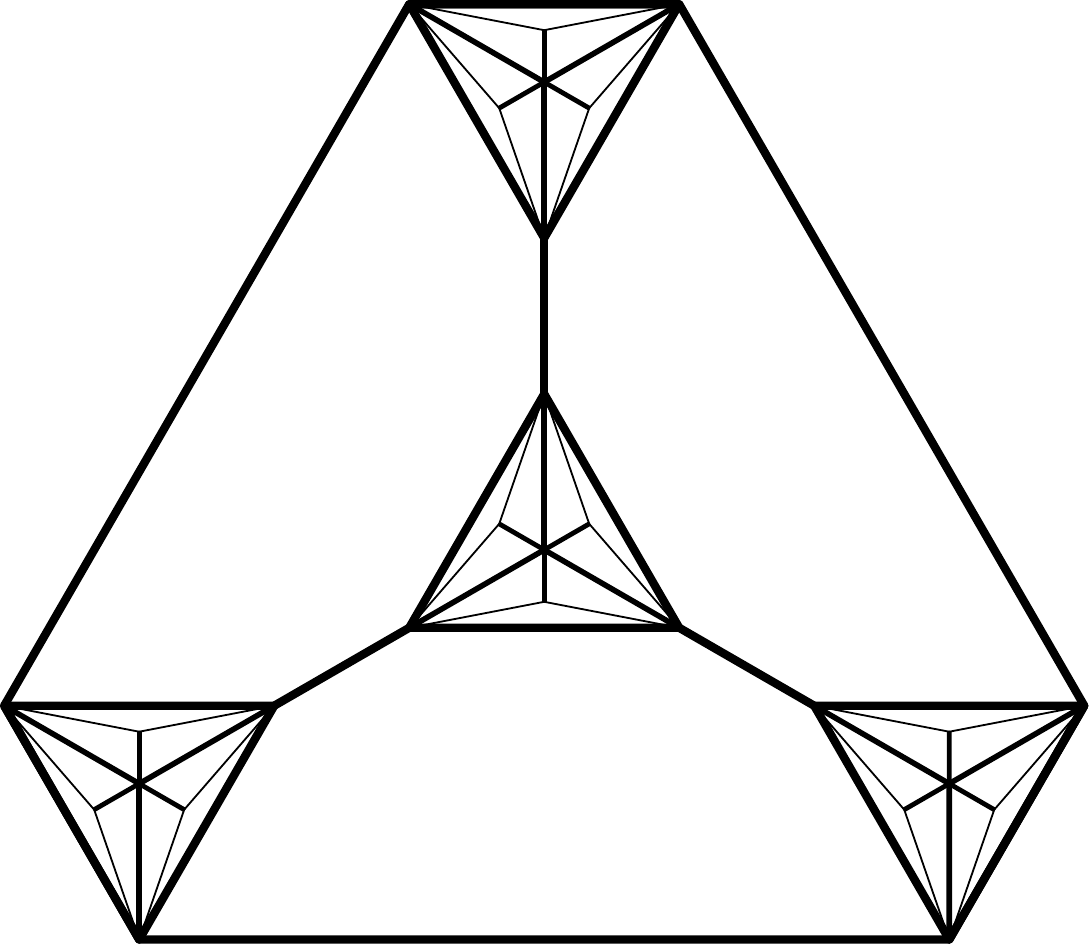}
	\caption{This polytope is not weakly circumscribable.}\label{fig:2stsimplex}
\end{figure}

\begin{proof}
	We start by showing that $\cP$ is not strongly circumscribable. First of all,
	the truncated tetrahedron is not strongly circumscribable (it is polar to the
	triakis tetrahedron).  For any facet arising from the truncation, we can
	replace the supporting half-space by any supporting half-space of one of the
	simplices stacked on the facet, and the result is still a truncated
	tetrahedron.  Hence $\cP$ is not strongly circumscribable.

	Now assume that $\cP$ is weakly circumscribed but not strongly circumscribed.
	Then there is a facet $F$ whose supporting half-space does not contain
	$\theBall$.  Every facet is incident to a vertex
	of degree $6$.  Let $v$ be such a vertex incident to $F$, then the vertex
	figure at $v$ is a weakly circumscibed hexagon that is not strongly
	circumscribed, contradicting Lemma~\ref{lem:noweaklypentagon}.
\end{proof}

The same method can be applied to many other $3$-polytopes proven to be
non-circumscribable by Steinitz (cf.\ \cite[Theorem 13.5.2]{grunbaum2003}) to
get an infinite family of $3$-polytopes that are not weakly circumscribable.
More specifically, if a simple $3$-polytope has more vertices than facets, then
truncating the vertices yields a polytope that is not strongly circumscribable,
and stacking twice on the truncated facets provides a polytope that is not
weakly circumscribable.

\subsection{Weak $(i,j)$-scribability}

We end by dealing with the remaining cases and showing that weak
$(i,j)$-scribability is indeed very weak.

\begin{reptheorem}{thm:weakij}
	Every $d$-polytope is weakly $(i,j)$-scribable for $0 \le i <j\le d-1$.
\end{reptheorem}

\begin{proof}
	It suffices to show that every $d$-polytope $\cP$ is weakly
	$(i,i+1)$-scribable for all $0\leq i\leq d-2$.  Consider a Euclidean
	realization of $\cP$ and a generic affine subspace $L$ of dimension $d-i$
	that does not intersect $\cP$.  Then $L$ intersects the affine span of every
	$(i+1)$-face of $\cP$ at a single point, but does not intersect the affine
	span of any $i$-face.  We can find, in the neighborhood of $L$, an ellipsoid
	$\cE$ that contains all these intersection points but remains disjoint from the
	affine span of every $i$-face. The affine transformation that sends $\cE$ to
	the unit ball $\theBall$ sends $\cP$ to a weakly $(i,i+1)$-scribed
	realization.
\end{proof}

\section{Stacked polytopes}\label{sec:stacked}
From now on, we focus only on strong scribability. In the remaining of the
paper,  we often omit the adjective ``\emph{strong}'', which has to be
understood whenever we talk about $(i,j)$-scribability unless explicitly stated
otherwise.

The goal of the current section is to solve the scribability problems for
stacked polytopes.  That is, we want to know which are the values of $k$ such
that every stacked $d$-polytope is $k$-scribable.

We recall the definitions of stacking and stacked polytope. For a $d$-polytope
$\cP$ with a simplicial facet $F$, we \defn{stack} a vertex onto $F$ by taking
the convex hull of $\cP\cup p$ for some point $p$ close enough to the
barycenter of $F$.  In terms of the \defn{connected sum} (cf.\
Section~\ref{sec:ridge}), this corresponds to gluing $\cP$ and a $d$-simplex
$\Delta$ along $F$.  A \defn{stacked polytope} is any polytope obtained from a
simplex by repeatedly stacking vertices onto facets.  The dual of a stacked
polytope is a \defn{truncated polytope}, obtained from a simplex by repeatedly
cutting off vertices.

A stacked polytope $\cP$ of dimension $d \ge 3$ has a unique triangulation $\cT$
with no interior faces of dimension $<d-1$, the \defn{stacked
triangulation}.  The \defn{dual tree} of $\cT$ takes the maximal simplices in
$\cT$ as vertices and connects two vertices if they share a face of dimension
$d-1$.  %

Already in dimension $3$, stacked polytopes provide the first examples of
polytopes that are not inscribable~\cite{steinitz1928}. Gonska and
Ziegler~\cite{gonska2013} proved that a stacked polytope is inscribable if and
only if all the nodes of its dual tree have degree~$\le 3$.  In higher
dimensions, Eppstein, Kuperberg and Ziegler~\cite{eppstein2003} proved that no
stacked $4$-polytope on more than $6$ vertices is edge-scribable.  

While it is well known that stacked polytopes present obstructions to
inscribability and edge-scribability, the other side of the story seems to have
escaped the attention of the community.  In this section, we show that stacked
polytopes are actually always circumscribable and ridge-scribable.  On the
other hand, stacked polytopes that are not $k$-scribable exist for any other
smaller $k$.

\begin{proposition}\label{prop:face}
	For any $d>3$ and $0\leq k\leq d-3$, there is a stacked $d$-polytope that is
	not $k$-scribable.
\end{proposition}

\begin{proof}
	In \cite{gonska2013} it is proved that, for every $d\geq 3$, there is a
	stacked $d$-polytope that is not inscribable, which solves the case $k=0$.
	We conclude the proposition by induction using Lemma~\ref{lem:property}(ii)
	and the fact that every stacked $d$-polytope is the vertex-figure of a
	stacked $(d+1)$-polytope, and a $k$-face figure of a stacked
	$(d+1+k)$-polytope. 
\end{proof}

We can strengthen this statement for the case $d>3$ and $k=d-3$.

\begin{proposition}\label{prop:subridge}
	For $d>3$, no stacked $d$-polytope with more than $d+2$ vertices is
	$(d-3)$-scribable.
\end{proposition}

\begin{proof}
	Any stacked $d$-polytope $\cP$ with more than $d+2$ vertices admits a vertex
	figure with the combinatorial type of a stacked $(d-1)$-polytope with more
	than $d+1$ vertices. To see this, first notice that the vertex figure at any
	vertex of $\cP$ has the combinatorial type of a stacked $(d-1)$-polytope.
	Since $\cP$ has more than $d+2$ vertices, it is not a simplex nor a
	bypiramid, hence the dual tree of $\cP$ has a node of degree~$\ge 2$.  Thus
	there are three simplices in the stacked triangulation sharing a ridge of
	$\cP$.  For any vertex in this ridge, the vertex figure contains at least
	$d+2$ vertices.	

	By Lemma~\ref{lem:property}(ii), if $\cP$ is $(d-3)$-scribable, its vertex
	figures are all $(d-4)$-scribable.  The theorem then follows by induction
	since no $4$-polytope on more than $6$ vertices is
	edge-scribable~\cite{eppstein2003}*{Corollary~9}. 
\end{proof}

\subsection{Circumscribability}

\begin{proposition}\label{prop:facet}
	Every stacked polytope is circumscribable.
\end{proposition}

\begin{proof}
	We prove by explicit construction the dual version of the proposition, namely
	that every truncated polytope is inscribable.

	Let $\cP$ be a truncated polytope, obtained from a $d$-simplex $\cP_0$ by
	repeatedly truncating vertices.  We start with an inscribed realization of $\cP_0$.

	In the first step, we perform simultaneously all the truncations that remove
	vertices of $\cP_0$.  This is carried out as follows:  For every vertex $v$
	to be truncated, we pull $v$ towards the exterior of $\theSphere$ by a
	sufficiently small distance.  Let $\cP'_0$ be the adjusted simplex.  Then the
	sphere $\theSphere$ intersects the edges of $\cP'_0$ near the adjusted
	vertices, while the other vertices remain on the sphere.  The convex hull of
	these intersection points gives the desired truncated polytope $\cP_1$ with
	all its vertices on $\theSphere$.  Observe that every vertex not present
	in $\cP_0$ is incident to a simplicial facet of~$\cP_1$ arising from the truncation.

	Now, let $\cP_k$ be the truncated polytope we obtain after the first $k$ steps. In the
	$(k+1)$-th step, we perform simultaneously all the truncations that remove
	vertices of $\cP_k$. The proof is by induction on $k$. We assume that all the
	vertices of $\cP_k$ are situated on the sphere $\theSphere$, and that every
	vertex $v$ not present in $\cP_{k-1}$ is incident to a simplicial facet $F_v$
	of $\cP_k$.  These assumptions have been verified for $k=1$.

	Consider all the vertices $v$ of $\cP_k$ to be truncated in the $(k+1)$-th
	step.  We pull every such $v$ by a sufficiently small distance towards the
	exterior of the sphere $\theSphere$ along the edge $e_v$ that is incident to
	$v$ but does not belong to $F_v$.  These movements do not change the
	combinatorial type of $\cP_k$.  Let $\cP'_k$ be the adjusted polytope.  Then
	the sphere $\theSphere$ intersects the $1$-skeleton of $\cP'_k$ near the
	adjusted vertices, while the other vertices remain on the sphere.  The convex
	hull of these intersection points gives the desired truncated polytope
	$\cP_{k+1}$ with all its vertices on~$\theSphere$, and every newly created
	vertex is incident to a simplical facet of $\cP_{k+1}$.

	It then follows by induction that $\cP$ has an inscribed realization. The
	procedure is sketched in Figure~\ref{fig:truncating}.
\end{proof}

\begin{figure}[hptb]
	\includegraphics[width=\linewidth]{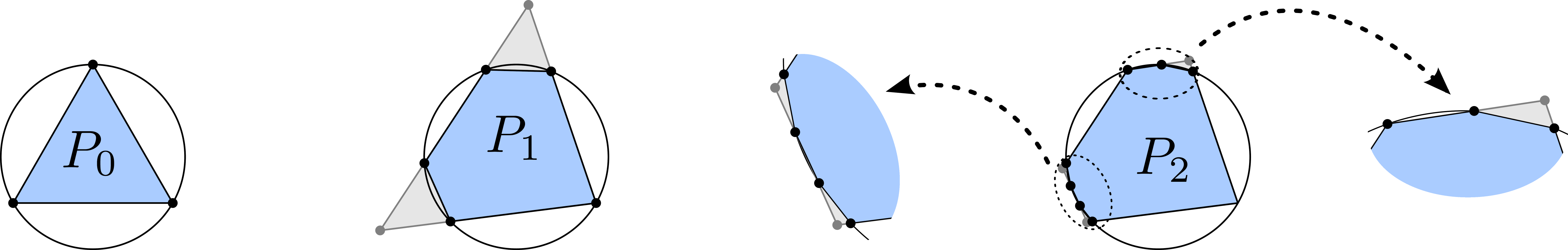}
	\caption{Steps of the proof of Proposition~\ref{prop:facet}}\label{fig:truncating} 
\end{figure}

The fact that every newly created vertex is incident to a simplex facet is
critical for this proof. It allows us to pull the vertices independently
without changing the combinatorial type of the polytope. %

Observe that this proof also works if we replace the ball by any strictly
convex body.  This means that truncated polytopes are \defn{universally
inscribable} in the sense of~\cite{gonskapadrol2016}.

\begin{proposition}\label{prop:egg}
	Every truncated polytope has a realization with all its vertices 
	on the boundary of any given smooth strictly convex body.%
\end{proposition}

And the construction also works in more general settings. For example, it works if
$\cP_0$ is an inscribable simple polytope (e.g.\ the cube) and every vertex of
$\cP_0$ is truncated in the first step.  In this case, the first step can be
carried out by shrinking the sphere by a sufficiently small amount. From this
moment, every vertex is adjacent to a a simplex facet and the remaining steps
remain as described in the proof.

\subsection{Ridge-scribability}\label{sec:ridge}

\begin{proposition}\label{prop:ridge}
	Every stacked polytope is ridge-scribable.
\end{proposition}

\begin{proof}
	Let $\Delta$ be a ridge-scribed realization of the simplex.  The interior of
	$\theBall$ can be regarded as the Klein model of $d$-dimensional
	hyperbolic space.  The tangency points of the ridges are all ideal in
	hyperbolic space, so the facets are all parallel or ultraparallel. Then the
	hyperbolic reflections in the facets of $\cP$ generate a universal Coxeter
	group $W$.  The associated Coxeter diagram is a complete graph with label
	$\infty$ on all the edges.  The fundamental domain of $W$ is $\Delta$.
	See~\cite{vinberg1967, vinberg1971} for more details on hyperbolic Coxeter
	groups.  Simplices in the orbit $W(\Delta)$ form a simplicial complex called
	the \emph{Coxeter complex}; see~\cite{abramenko2008}.
	
	Stacked polytopes can be seen as strongly connected subcomplexes of the
	Coxeter complex.  To see this, notice that for any $w\in W$, the simplex
	$w(\Delta)$ is a ridge-scribed simplex in the Euclidean view.  The dual graph
	of $W(\Delta)$ is the Cayley graph of $W$, which is a tree.  The convexity is
	guaranteed by the fact that, for any hyperplane $H$ that is tangent to
	$\theBall$ and contains a ridge of $W(\Delta)$, $W(\Delta)$ is contained in
	the halfspace $H^-$.  Ridge-scribed realizations for stacked polytopes are
	therefore given by the Coxeter complex. 
\end{proof}

Inspired by this proof, we now extend the proof to a generalization of stacked
polytopes.  Let $\cP$ be a ridge-scribed polytope.  The reflections in the
facets of $\cP$ again form a hyperbolic Coxeter group.  The Coxeter complex is
a polytopal cell complex, each cell being a copy of $\cP$.  Every strongly
connected subcomplex of the Coxeter complex again forms a convex polytope,
which we call a \defn{stacked $\cP$-polytope}.  Then the same argument proves
that

\begin{proposition}
	Stacked $\cP$-polytopes are ridge-scribable if $\cP$ is.
\end{proposition}

We can further extend the proof to connected sums of polytopes.  Recall that
two polytopes are \defn{projectively equivalent} if there is a projective
transformation that sends one to the other.  Let $\cP$ and $\cQ$ be two
polytopes, each with a facet projectively equivalent to~$F$, then the
\defn{connected sum} of $\cP$ and $\cQ$ through $F$, denoted $\cP \#_F \cQ$, is
obtained by ``gluing'' $\cP$ and $\cQ$ by identifying the projectively
equivalent facets (see \cite[Section~3.2]{richter-gebert1996}).  So the
operation of stacking is actually taking connected sum with a simplex.

We say that two polytopes are \defn{M\"obius equivalent} if there is a
M\"obius transformation (projective transformation preserving $\theSphere$)
that sends one to the other.  Then

\begin{proposition}\label{prop:consum}
	Let $\cP$ and $\cQ$ be ridge-scribed polytopes, each with a facet M\"obius
	equivalent to~$F$, then the connected sum $\cP \#_F \cQ$ is ridge-scribable.
\end{proposition}

\begin{figure}[htb]
	\includegraphics[width=.7\textwidth]{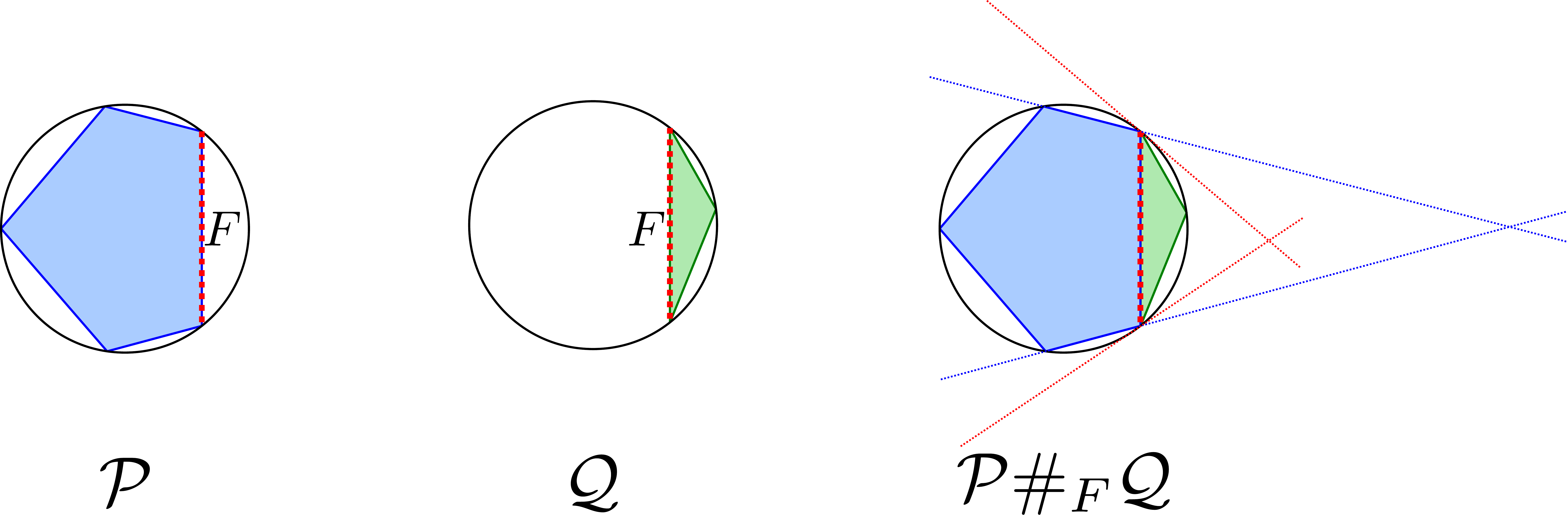}
	\caption{
		The relevant tangent hyperplanes and facet defining hyperplanes in a
		connected sum of ridge-scribed polytopes through a M\"obius equivalent
		facet.\label{fig:ConnectedSum}
	}
\end{figure}

\begin{proof}
	With a M\"obius transformation if necessary, we may assume that $\cP$ and
	$\cQ$ are ridge-scribed and $\cP\cap \cQ=F$.  For any ridge $R$ adjacent to
	$F$, the hyperplane that is tangent to $\theSphere$ and contains $R$ is
	supporting both for $\cP$ and for $\cQ$ (see the proof of
	Lemma~\ref{lem:equivalent}).  So the polytope $\cP \#_F \cQ = \cP \cup \cQ$
	is convex and ridge-scribed by construction.  
\end{proof}

Ridge-scribed simplices are all M\"obius equivalent;
see~\cite[Lemma~7]{eppstein2003} for the dual statement.  We can therefore
regard Proposition~\ref{prop:ridge} as a special case of
Proposition~\ref{prop:consum}.

Finally, we prove an interpretation of Proposition~\ref{prop:ridge} in terms of
ball packings. For every point $\bx\in\Euc\setminus\theBall$, the part of
$\theSphere$ visible from $\bx$ is a spherical cap on $\theSphere$.  For an
edge-scribed polytope, the caps corresponding to the vertices have disjoint
interiors.  After a stereographic projection, they form a ball packing in
Euclidean space whose tangency graph is isomorphic to the $1$-skeleton of the
polytope; see~\cite{chen2016}.  The dual version of
Proposition~\ref{prop:ridge} says that every truncated polytope is
edge-scribable.  Therefore,

\begin{corollary}
	The $1$-skeleton of every truncated $d$-polytope is the tangency graph of a
	ball packing in dimension $d-1$.
\end{corollary}

Here we provide a self-consistent proof independent to
Proposition~\ref{prop:ridge}.  In fact, the two proofs are essentially the
same, as inversions correspond to hyperbolic reflections; see~\cite{chen2015}.

\begin{figure}[htpb]
	\includegraphics[width=.95\textwidth]{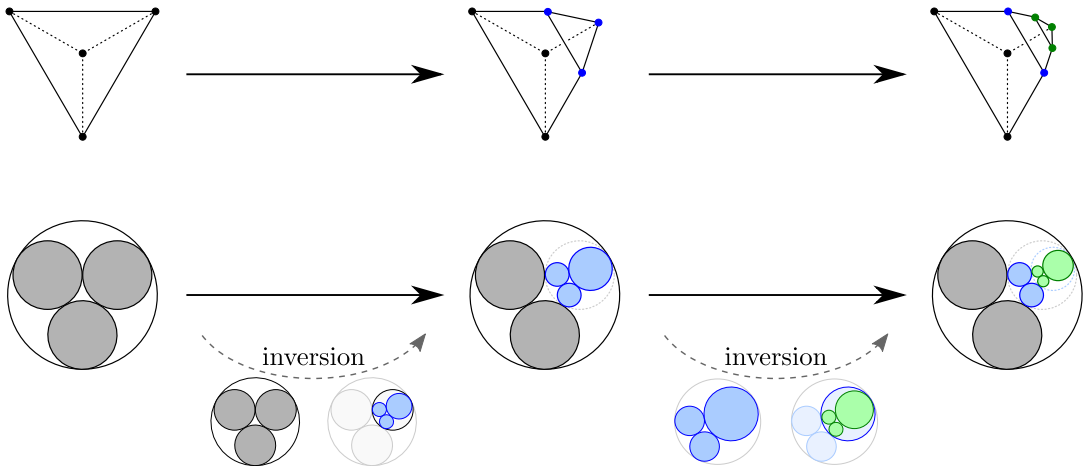}
	\caption{
		The graph of a truncated polytope as a ball packing through a series of
		inversions.\label{fig:BallPackings}}
\end{figure}

\begin{proof}
	We construct the truncated polytope and the ball packing in parallel, and
	keep the $1$-skeleton and the tangency graph isomorphic.  We begin with an
	edge-tangent $d$-simplex $\Delta$ and the corresponding configuration of
	$d+1$ pairwise tangent balls $B_0, \dots, B_d$ of dimension $d-1$.  
	
	We now proceed by induction, truncating the vertices one by one. Let $B$ be
	the ball corresponding to the vertex to be truncated.
	
	If $B$ is in the initial packing of $d+1$ balls, say $B=B_0$, then the
	inversion in $\partial B$ sends $B_1,\dots,B_d$ into pairwise tangent balls
	$B'_1,\dots,B'_d$ in the interior of $B$, and $B'_i$ is tangent to $B_i$, $1
	\le i \le d$.  Since the tangency points are preserved, replacing $B$ by
	these primed balls gives the desired packing. 
	
	Otherwise, $B$ appeared in some latter truncation step, when a ball $B'_d$
	was replaced by $d$ pairwise tangent balls $B'_0,\dots,B'_{d-1}$. Say
	$B=B'_0$.  They are all tangent from the interior to $\partial B'_d$, which
	we can regard the exterior of $B'_d$ as a ball of negative curvature.  The
	inversion in $\partial B$ sends $B'_1,\dots,B'_d$ into pairwise tangent balls
	$B''_1,\dots,B''_d$ in the interior of $B$. Again, the tangency points of
	these balls with $B$ coincide with those of $B$ with its neighbors. Replacing
	$B$ by these doubly primed balls gives the desired packing.

	The corollary then follows by induction. See also Figure~\ref{fig:BallPackings}.
\end{proof}

\subsection{The $(i,i+1)$ scribability}

Finally, we prove Theorem~\ref{thm:stack01} by constructing stacked $d$-polytopes
that are not $(i,i+1)$-scribed for $0 \le i \le d-4$.

We still regard the interior of $\theBall \subset \Euc^d$ as the Klein model
of the hyperbolic space $\Hyp^d$.  To study the vertex figure $\cP/v$, consider
a surface $\Sigma$ that intersects perpendicularly all the hyperbolic geodesics
that pass through $v$ (in $\Euc^d$).  There are three cases: If $v$ is a
point of $\Hyp^d$ (in the interior of~$\theBall$), $\Sigma$ is a $(d-1)$-sphere
$\Sph^{d-1}$ centered at $v$; this is more clear if we assume $v$ is the center
of~$\theBall$, or simply use the Poincar\'e ball model for $\Hyp^d$.  If $v$ is
an ideal point of $\Hyp^d$ (on the boundary of~$\theBall$), then $\Sigma$ is a
horosphere based at $v$, which can be identified to the Euclidean space
$\Euc^{d-1}$; this can be easily seen with Poincar\'e half-space model,
cf.~\cite[\S~6.4]{ratcliffe2006}.  Finally, if $v$ is hyperideal for $\Hyp^d$
(in the exterior of~$\theBall$), then $\Sigma$ is the totally geodesic surface
given by the intersection $\theBall \cap v^\perp$, which can be identified to
the hyperbolic space $\Hyp^{d-1}$.

Now the vertex figure $\cP/v$ can be realized as the intersection of $\Sigma$
and the cone over $\cP$ with apex $v$, which is a polytope in $\Sph^{d-1}$,
$\Euc^{d-1}$ or $\Hyp^{d-1}$ if $v$ is in the interior, boundary or exterior 
of~$\theBall$, respectively.  Since $\Sigma$ is perpendicular to all the
hyperbolic geodesics through $v$, the dihedral angles of $\cP$ are preserved in
$\cP \cap \Sigma$; cf.~\cite{ratcliffe2006}.  More generally, a face figure
$\cP/F$ can be obtained by consecutively taking vertex figures at each vertices
of $F$.  So we can realize $\cP/F$ as a spherical, Euclidean or hyperbolic
polytope if $F$ is (in the strong sense) strictly cutting, tangent or strictly
avoiding $\theSphere$, respectively.

We will need the following lemma:

\begin{lemma}\label{lem:dihedral}
	Let $\cP$ be a $(0,d-3)$-scribed $d$-simplex, and $F$ be a facet of $\cP$.
	If we regard~$\theBall$ as the Klein model of the hyperbolic space $\Hyp^d$,
	then the hyperbolic dihedral angles at the ridges incident to $F$ sum up to
	at least $\pi$.
\end{lemma}

\begin{proof}
	Since the $(d-3)$-faces cut $\theSphere$, their links are spherical or
	Euclidean triangles, i.e.\ the dihedral angles at the ridges incident to a
	$(d-3)$-face sum up to at least $\pi$.  If we consider only the $(d-3)$-faces
	incident to $F$, the dihedral angles at the ridges incident to them sum up to
	at least ${d \choose 2}\pi$.
	
	However, this summation also includes some ridges not incident to $F$.  These
	ridges are all incident to the vertex $v$ that is not in $F$.  Since $v$
	avoids $\theSphere$, the vertex figure is a hyperbolic or Euclidean
	$(d-1)$-simplex.  By a result of H\"ohn~\cite{hohn1953} (see
	also~\cite{gaddum1956}) and~\cite{au2008}), the dihedral angles of these
	ridges sum up to at most ${d-1 \choose 2}\pi$.  Subtracting this from the
	summation above yields at least ${d \choose 2}\pi - {d-1 \choose 2}\pi =
	(d-1) \pi$.

	Furthermore, every ridge is counted $d-1$ times, which is the number of
	$(d-3)$-faces incident to each ridge.  So the sum of the dihedral angles at
	the ridges incident to $F$ sum up to at least~$\pi$.
\end{proof}

By
Lemma~\ref{lem:property}, Theorem~\ref{thm:stack01} is derived from the
following proposition.

\begin{proposition}\label{prop:stack01}
	There is a stacked $4$-polytope that is not $(0,1)$-scribable.
\end{proposition}

\begin{proof}
	Consider the stacked $4$-polytope obtained by stacking on every facet of a
	simplex, and then stacking again on every facet.  The stacked triangulation
	of the resulting polytope consists of $1+5+20=26$ $4$-simplices.  Twenty of
	them have boundary facets, we call them \defn{exterior simplices}.  The remaining
	six only have interior facets, we call them \defn{interior simplices}.

	There are $40$ ridges incident to the interior facets, and we want to
	estimate the sum of the hyperbolic dihedral angles at these ridges.  Each
	interior simplex contributes at least $10\pi/3$.  To see this, notice that
	the link of each edge is a spherical or Euclidean triangle, and so the
	adjacent dihedral angles sum up to at least $\pi$, and that each ridge is
	incident to $3$-edges, so each angle is counted three times.  On the other
	hand, each exterior simplex shares a facet with an interior simplex. Hence,
	by Lemma~\ref{lem:dihedral}, it contributes with at least $\pi$ to the sum of
	the dihedral angles of ridges incident to interior simplicies.  Hence, the
	dihedral angles at these $40$ ridges sum up to at least $40\pi$.

	Therefore, there is at least one ridge at which the hyperbolic dihedral angle
	is at least $\pi$, which destroys the convexity of the polytope.
\end{proof}

Bao and Bonahon~\cite{bao2002} characterized the dihedral angles of
$(0,1)$-scribed $3$-polytopes (hyperideal polyhedra) and proved that the
dihedral angles determine a polytope uniquely up to hyperbolic isometry; see
also~\cite{schlenker2005} for the connection to circle configurations.

\section{Cyclic polytopes}\label{sec:neighborly}
The main result of this section is that even-dimensional cyclic polytopes with
sufficiently many vertices are not $(1,d-1)$-scribable.  We also investigate
odd-dimensional cyclic polytopes and neighborly polytopes in general. 

A $d$-polytope is \defn{$k$-neighborly} if every $k$ vertices form a face.
Since the only $k$-neighborly $d$-polytope with $k>\lfloor d/2 \rfloor$ is the
simplex, we call a $d$-polytope simply \defn{neighborly} if it is $\lfloor d/2
\rfloor$-neighborly. 

The most important examples of neighborly polytopes are cyclic polytopes.
Consider a curve~$\gamma$ of \defn{order~$d$}, which means that each hyperplane
intersects $\gamma$ in at most $d$ points, such as the $d$-dimensional moment
curve \((t,t^2,\dots,t^d)\).  Take the convex hull of $n$ distinct points on
$\gamma$.  That is,
\[
	\conv(\gamma(t_1),\gamma(t_2),\dots,\gamma(t_n))
\]
for $n$ distinct parameters $t_1<t_2<\dots<t_n$.  Then the combinatorial type
of this polytope (and, even more, the oriented matroid defined by the points
$\gamma(t_i)$) does not depend on the choice of the parameters~$t_i$. We call
any polytope of this combinatorial type a \defn{cyclic $d$-polytope} with $n$
vertices, and denote it by $\cyc{d}{n}$.  If we identify the vertices of
$\cyc{d}{n}$ with the indices $[n]=\{1,\dots,n\}$, the combinatorics of
$\cyc{d}{n}$ are described by the following criterion, called \defn{Gale's
evenness condition} (cf. \cite[Section~4.7]{grunbaum2003},
\cite[Theorem~0.7]{Ziegler1995}).

\begin{proposition}\label{prop:evenness}
	Let $I\subset[n]$ be a set of $d$ vertices. Then $I$ indexes a facet of
	$\cyc{d}{n}$ if and only if for any two vertices $j<k$ in $[n]\setminus I$,
	the set $\{i\in I\mid j<i<k\}$ contains an even number of vertices.
\end{proposition}

It is well-known that cyclic polytopes are inscribable;
see~\cite{caratheodory1911} \cite[p.~67]{grunbaum1987}
\cite[p.~521]{seidel1991} \cite[Proposition~17]{gonska2013}. This implies that
they are $(0,j)$-scribable for any $j\geq 0$.  We will however see that, in
even dimensions, cyclic polytopes provide examples that are not
$(i,j)$-scribable for each $i>0$.  In particular, cyclic polytopes behave
poorly with respect to $k$-scribability, as indicated by
Theorem~\ref{thm:cyclic}, which we recall below.

\begin{reptheorem}{thm:cyclic}
	For any $d>3$ and $1\le k\le d-1$, a cyclic $d$-polytope with sufficiently
	many vertices is not $k$-scribable.
\end{reptheorem}

This follows from the main results of this section, namely
Propositions~\ref{prop:even} for even dimensions, Corollary~\ref{cor:2d-1} for
odd dimensions and $k>1$, and Proposition~\ref{prop:neighborly} for $k=1$. 

\subsection{$k$-ply systems and $k$-sets}

As we have already mentioned in Section~\ref{sec:stacked}, any point $\bx$ in
$\Euc^d \setminus \theBall$ can be associated with a closed spherical cap on
$\theSphere$, namely the set of points of $\theSphere$ that are visible from
$\bx$. A set of spherical caps on $\theSphere$ is said to be a \defn{$k$-ply
system} if no point of $\theSphere$ is in the interior of $k$~caps.  These
systems were studied by Miller et al.~\cite{miller1997}, who proved the
following Separation Theorem.  Here, the \defn{intersection graph} is the graph
where every vertex represents a cap, and two caps form an edge if they
intersect.

\begin{proposition}[Sphere Separator Theorem]\label{prop:separator}
	The intersection graph of a $k$-ply system consisting of $n$ caps on a
	$d$-dimensional sphere can be separated into two disjoint parts, each of size
	at most $\frac{d+1}{d+2}n$, by removing $O(k^{1/d}n^{1-1/d})$ vertices.
\end{proposition}

To the knowledge of the authors, the best known constant factor in the
proposition is
\[
	c_2=\sqrt{\frac{2\pi}{\sqrt 3}}\Bigg(\frac{1+\sqrt k}{\sqrt{2(1+k)}}+o(1)\Bigg)
\]
	for $d=2$ and
\[
	c_d=\frac{2A_{d-1}}{A_d^{1-1/d}V_d^{1/d}}+o(1)
\]
for $d>2$; see~\cite{spielman1996}.  Here $V_d$ is the volume of a unit
$d$-ball and $A_d$ is the area of a unit $d$-sphere, so $A_{d-1}=dV_d$.

For a point set $V\subset \Euc^d$, a subset $I$ of cardinality $k$ is said to
be a \defn{$k$-set} if there is a hyperplane strictly separating $I$ and
$V\setminus I$.  We will define the $k$-sets of a polytope to be the $k$-sets
of its set of vertices, and say that a $k$-set intersects $\theSphere$ if its
convex hull intersects $\theSphere$. The following lemma relates $k$-sets and
$k$-ply systems.

\begin{lemma}\label{lem:kply}
	A point set $V \subset \Euc^d \setminus \theBall$ corresponds to a $k$-ply
	system on $\theSphere$ if and only if every $k$-set intersects the sphere
	$\theSphere$.
\end{lemma}

\begin{proof}
	Assume that there is a $k$-set $I$ such that $\conv I \cap \theSphere =
	\emptyset$.  Then there is a hyperplane tangent to $\theSphere$ separating
	$I$ and the interior of $\theBall$.  The tangency point is visible from every
	point in $I$.  In other words, it is in the interior of at least $k$ of the
	associated caps, so the set of caps corresponding to $V$ is not a $k$-ply
	system.  The other direction is obtained by reversing the argument.
\end{proof}

The following obvious fact can be regarded as a special case of this lemma.

\begin{corollary}\label{cor:edges}
	The caps of~$\theSphere$ corresponding to $v,w\in \Euc^d\setminus\theBall$
	have disjoint interiors if and only if the segment $[v,w]$ strongly
	cuts~$\theSphere$.
\end{corollary}

\subsection{Even dimensional cyclic polytopes}

The following is the key for proving our main result. It uses that all even-dimensional cyclic polytopes 
have the same oriented matroid to make statements about the $k$-sets of any realization of $\cyc{d}{n}$ 
(which fail for odd dimensions, cf. Section~\ref{sec:oddcyclic}).

\begin{lemma}[$k$-set Lemma]\label{lem:kset}
	For even $d$ and $k\ge 3d/2-1$, every $k$-set of $\cyc{d}{n}$ contains a
	facet of $\cyc{d}{n}$. 
\end{lemma}

\begin{proof}
	Without loss of generality, we can assume that $k=3d/2-1$.  Let $I$ be a
	$k$-set of $\cyc{d}{n}$.

	Every even-dimensional cyclic polytope has its vertices on an order-$d$ curve
	$\gamma$~\cite{Sturmfels1987}. Every hyperplane $H$ intersects $\gamma$ in at
	most $d$ points, hence there can be at most $d$ changes of sides of $H$
	between $I$ and $[n]\setminus I$.  So $I$ is decomposed into at most $d/2+1$
	consecutive segments of~$[n]$.
	
	We call a consecutive segment of $[n]$ \emph{external} if it contains $1$ or
	$n$, or \emph{internal} otherwise. If there are $c$ changes of sides, $c \le
	d$, then $I$ has either $\lfloor c/2 \rfloor$ or $\lceil c/2 \rceil - 1$
	internal segments.  In the worst case with $c=d$, there are $d/2$ internal
	segments; they can not be all of odd length because $d$ is even and $k=3d/2-1
	\not\equiv d/2 \pmod 2$.  We then conclude that at most $d/2-1$ internal
	segments have odd lengths.
	
	By removing a vertex from the boundary of each of the odd internal segments,
	we obtain a set $J$ satisfying Gale's evenness condition except for the size.
	In the worst case, $d/2-1$ internal segments are odd, hence $J$ contains at
	least $k+1-d/2 \ge d$ vertices.  This allows us to take a $d$-element subset
	of $J$ by taking even subsegments from the internal segments, together with
	external subsegments from the external segments.  This set corresponds to a
	facet since it still fulfills Gale's evenness condition.
\end{proof}

\begin{figure}[htpb]
	\includegraphics[width=.7\linewidth]{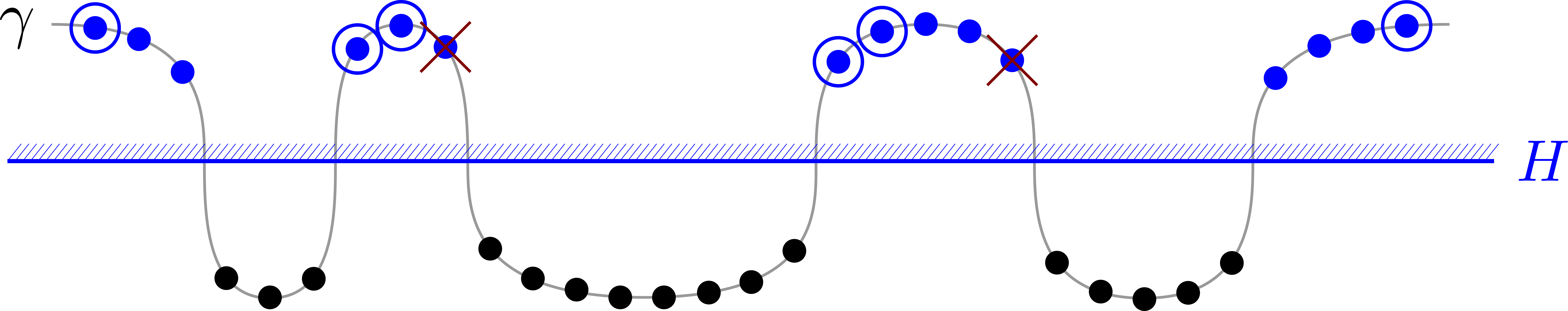}
	\caption{
		Sketch of a $15$-set of a cyclic $6$-polytope. The curve~$\gamma$ intersects
		the hyperplane~$H$ in $6$ points, separating them  into the $k$-set (above)
		and its complement (below).
	}\label{fig:kset}
\end{figure}

The proof of the $k$-set Lemma is illustrated in Figure~\ref{fig:kset}. It shows a
$15$-set of a cyclic $6$-polytope, which consists of $4$ segments, of lengths $3$,
$3$, $5$ and $4$, respectively. The second and third segments are internal. We
remove one extreme point from every internal odd segment (marked with a cross).
Then we can select a subset that forms a facet (circled elements).

We are finally ready to prove the main result of this part. 
\begin{proposition}\label{prop:even}
	Let $d\ge 4$ be an even integer and 
	\[
		n>(c_{d-1}(d+1))^{d-1}(3d/2-1),
	\] 
	then the cyclic polytope $\cyc{d}{n}$ is not $(1,d-1)$-scribable.
\end{proposition}

By Lemma~\ref{lem:property}(i), this implies that, in even dimensions,
$\cyc{d}{n}$ is not $(i,j)$-scribable for $1 \le i \le j \le d-1$ if $n$ is
large enough.  In particular, we have proved Theorem~\ref{thm:cyclic} for even
dimensions.

\begin{proof}
	Assume that $\cP$ is a $(1,d-1)$-scribed cyclic
	polytope with $n$ vertices, for a value of $n$ larger than $(c_{d-1}(d+1))^{d-1}(3d/2-1)$.  Let $k=3d/2-1$. By the $k$-set
	Lemma~\ref{lem:kset}, every $k$-set of the vertices of $\cP$ contains a
	facet. Since every facet of $\cP$ cuts the sphere $\theSphere$, this implies
	that every $k$-set intersects $\theSphere$.  Hence, the collection spherical
	caps corresponding to the vertices of $\cP$ form a $k$-ply system.  By the
	Sphere Separator Theorem~\ref{prop:separator}, the intersection graph of the
	caps admits a separator of size
	\[
		c_{d-1}\lfloor d/2\rfloor^\frac{1}{d-1}n^\frac{d-2}{d-1}<\frac{n}{d+1}.
	\]
	However, since all the edges of $\cP$ strongly avoid the sphere, the
	intersection graph is a complete graph by Lemma~\ref{cor:edges}, and the
	removal of the separator leaves a complete graph of more than
	$\frac{d}{d+1}n$ vertices, contradicting Proposition~\ref{prop:separator}.
\end{proof}

By Lemma~\ref{lem:property}(iii) and (iv), we obtain the following corollary,
which provides the final counterexamples to $(i,j)$-scribability for the proof
of Theorem~\ref{thm:strongij}.
\begin{corollary}\label{cor:1d-2}
	For odd $d \ge 5$, the pyramid over a cyclic $(d-1)$-polytope with sufficiently
	many vertices is a $d$-polytope that is neither $(1,d-2)$-scribable nor
	$(2,d-1)$-scribable.
\end{corollary}

\subsection{Odd dimensional cyclic polytopes}\label{sec:oddcyclic}

The proof of the $k$-set Lemma~\ref{lem:kset} fails dramatically in odd
dimensions.  When $d$ is odd, different realizations of $\cyc{d}{n}$ may have
different oriented matroids and hence different $k$-set structures.  In
particular, the vertices do not necessarily lie on any order-$d$ curve.  In
fact, the $k$-set Lemma~\ref{lem:kset} does not hold in odd dimensions, as we
will see in Remark~\ref{rmk:oddkset}.

Nevertheless, since $\cyc{d}{n}$ has $\cyc{d-1}{n-1}$ as a vertex-figure, we
obtain the following corollary by Lemma~\ref{lem:property}(iv).

\begin{corollary}\label{cor:2d-1}
	Let $d\ge 5$ be an odd integer and 
	\[
		n>(c_{d-1}(d+1))^{d-1}(\lfloor 3d/2 \rfloor - 1),
	\] 
	then the cyclic polytope $\cyc{d}{n}$ is not $(2,d-1)$-scribable.
\end{corollary}

By Lemma~\ref{lem:property}(i), this implies that, in odd dimensions,
$\cyc{d}{n}$ is not $(i,j)$-scribable with $2 \le i \le j \le d-1$ if $n$ is
large enough. This proves Theorem~\ref{thm:cyclic} for $1<k\le d-1$.  The
$1$-scribability of odd-dimensional cyclic polytopes will be taken care of in
Section~\ref{ssec:neighborly}.

However, odd-dimensional cyclic polytopes are $(1,d-1)$-scribable, in contrast
to the situation in even dimensions.

\begin{proposition}\label{prop:odd}
	For odd $d$, the cyclic polytope $\cyc{d}{n}$ is $(1,d-1)$-scribable.
\end{proposition}

\begin{proof}
	Assume that $\theSphere$ is the sphere of unit radius centered at the origin
	of $\RR^d$.  Let $\bx=(0,0,\dots,0,-1)$.  We
	place a light source at $\bx_+=(h,0,\dots,0,-1)$, where $h$ is a
	variable height that will be decided later.  When $h>1$, the boundary of 
	the shadow of $\theSphere$ on
	the hyperplane $x_0=0$ is an ellipsoid $\cE \in \RR^{d-1}$ that contains
	$\bx$.
	
	At the beginning, we take $h=\infty$, so the shadow is the equator of
	$\theSphere$ with $x_0=0$.  Take a realization of $\cyc{d-1}{n-1}$ in the
	hyperplane $x_0=0$ that is inscribed with respect to the equator, 
	with one vertex $v$ placed at $\bx$. Now we decrease the height $h$ of the light
	source and, at the same time, move the vertices of $\cyc{d-1}{n-1}$ with the shadow, 
	so that they remain on the boundary of~$\cE$.  The convex
	hull of these vertices is still a realization of $\cyc{d-1}{n-1}$ (we just
	did a projective transformation) that is
	inscribed with respect to the shadow $\cE$.
	
	We keep decreasing the height $h$ until every edge of $\cyc{d-1}{n-1}$ that
	is disjoint from $v$ avoids the equator of $\theSphere$. This can always be done
	if the original realization of $\cyc{d-1}{n-1}$ was chosen wisely. For example,
	by taking all the vertices but $v$ very close together in the trigonometric moment curve.
	Then the convex
	hull of $\bx_+$, $\bx_-=(-h,x_1,\dots,x_{d-1})$ and the vertices of
	$\cyc{d-1}{n-1}$ other than $v$, gives a realization of $\cyc{d}{n}$ (cf.\
	\cite{cordovil2000} and Remark~\ref{rmk:CordovilDuchet}).  We denote the
	vertex at $\bx_+$ (resp.\ $\bx_-$) by $v_+$ (resp.\ $v_-$).
	
	In this realization, edges incident to $v_\pm$
	are tangent to $\theSphere$ by construction; 
	the other edges, belonging to $\cyc{d-1}{n-1}$,
	avoid $\theSphere$ also by construction.  On the other hand, all the facets are
	adjacent to either $v_+$ or $v_-$ and contain a tangent edge, and hence cut~$\theSphere$. The construction is sketched in Figure~\ref{fig:oddcyclic}.
\end{proof}

\begin{figure}[htpb]
	\includegraphics[width=.7\linewidth]{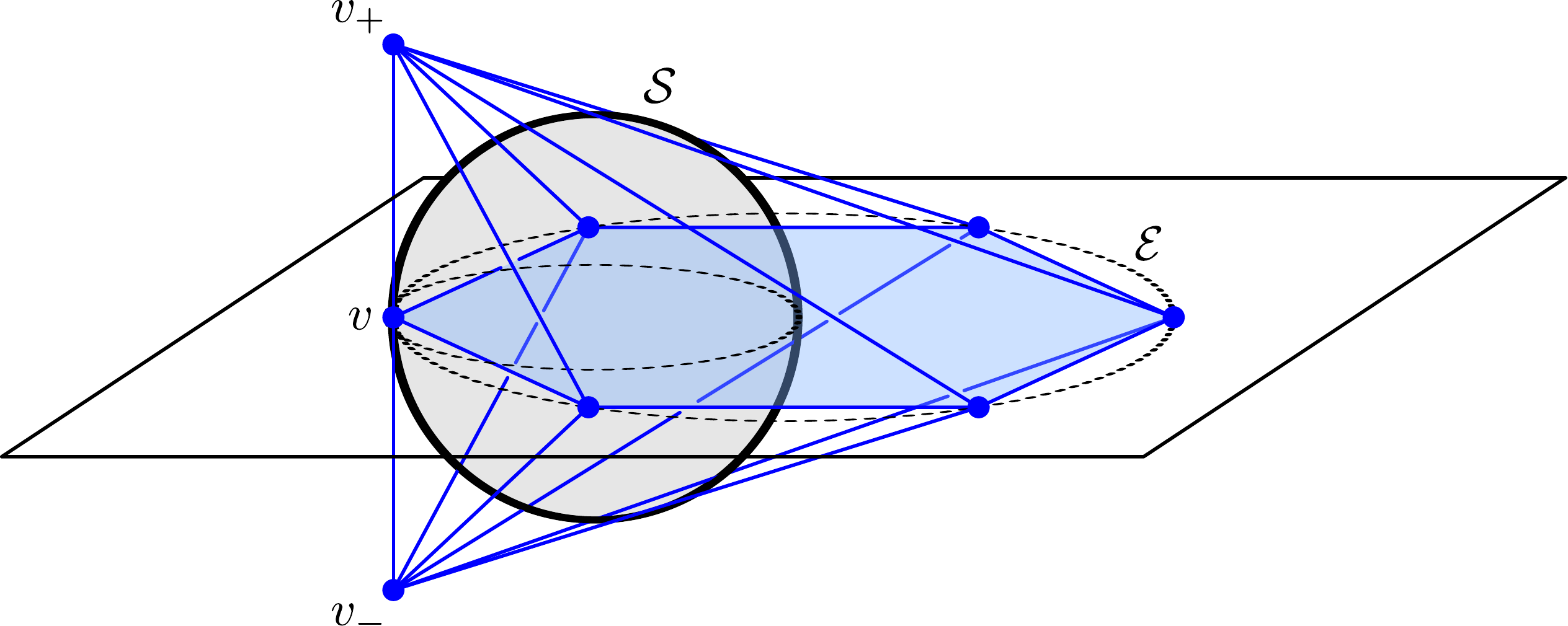}
	\caption{Sketch of the construction of Proposition~\ref{prop:odd}.}\label{fig:oddcyclic}
\end{figure}

\begin{remark}\label{rmk:CordovilDuchet}
	Cordovil and Duchet~\cite{cordovil2000} proposed a process that can realize
	any oriented matroid for $\cyc{d}{n}$ when $d$ is odd, but their description
	does not quite work.  They first stack a vertex onto $\cyc{d-1}{n-1}$, then
	split it into an extra dimension, followed by a perturbation.  This process
	does not in general give a cyclic polytope.  The correct construction
	consists of splitting a vertex of $\cyc{d-1}{n-1}$ into an extra dimension,
	as we did in the proof above, and then a perturbation.
\end{remark}

\begin{remark}\label{rmk:oddkset}
There can be arbitrarily large $k$-sets of an odd dimensional cyclic polytope that do not contain a facet. To see this, take
the realization of $\cyc{d}{n}$ from Proposition~\ref{prop:odd}. From the copy of $\cyc{d-1}{n-1}$ lying in $x_0=0$ take a subset of vertices not containing any facet, and lift them to height $x_0=\varepsilon>0$; and descend the remaining vertices of 
$\cyc{d-1}{n-1}$ to $x_0=-\varepsilon$. If $\varepsilon$ is sufficiently small, then this does not change the combinatorial type,
and the points in the open half-space $x_0>0$ form a $k$-set not containing any facet. 
\end{remark}

\subsection{Neighborly polytopes}\label{ssec:neighborly}

In this part, we apply the ideas leading to Proposition~\ref{prop:even} to
general neighborly polytopes. However, the lack of an analogue to
Lemma~\ref{lem:kset} does not allow us to carry over the argument in full
generality.

Let $\cP$ be a $j$-neighborly $d$-polytope.  Since every $j$-set of $\cP$ forms
a $(j-1)$-face, the argument for Proposition~\ref{prop:even} proves that
if $\cP$ has sufficiently many vertices then it is not $(1,j-1)$-scribable.  In
particular, neighborly polytopes with sufficiently many vertices are not
edge-scribable. This provides the last missing piece (namely $k=1$) for
Theorem~\ref{thm:cyclic}.

We will however prove a slightly stronger result

\begin{proposition}\label{prop:neighborly}
	For $d\geq 4$, a $k$-neighborly $d$-polytope $\cP$ with sufficiently many
	vertices is not $(1,k)$-scribable.
\end{proposition}

\begin{remark}
	$d+3$ vertices suffice for a $1$-neighborly polytope to be not edge-scribable;
	see~\cite[Exercise~20.12]{pak2010}.
\end{remark}

For a proof, the argument for Proposition~\ref{prop:even} applies almost
directly.  But in place of Lemma~\ref{lem:kset}, we need the following $k$-set
lemma.

\begin{lemma}[$k$-set lemma for neighborly polytopes]\label{lem:ksetneighborly}
	Every $(k+1)$-set of a $k$-neighborly $d$-polytope is a $k$-face.
\end{lemma}

For a polytope, a set $I$ of vertices is a \defn{missing face} if $I$ is not a
face but every proper subset $I$ is.  For a $k$-neighborly polytope, every
$(k+1)$-set either forms a face or a missing face.  The $k$-set
Lemma~\ref{lem:ksetneighborly} is then a special case of the following more
general lemma.

\begin{lemma}\label{lem:missingface}
	The $k$-sets of a polytope are not missing faces.
\end{lemma}

\begin{proof}
	Let $V$ be the vertex set of the polytope, and $I$ be a subset of $k$
	vertices.  If $I$ is a $k$-set, then $\conv(I) \cap \conv(V \setminus I) \ne
	\emptyset$.  But if $I$ is a missing face, then $\conv(I) \cap \conv(V
	\setminus I) = \emptyset$.
\end{proof}

Neighborliness is a property that only depends on the $f$-vector.   Hence
Proposition~\ref{prop:neighborly} implies Theorem~\ref{thm:fvector}, which we
restate here:

\begin{reptheorem}{thm:fvector}
	For $d\geq 4$ and any $1 \le k \le d-2$, there are $f$-vectors such that no $d$-polytope
	with those $f$-vectors are $k$-scribable.
\end{reptheorem}

\begin{proof}
For $1 \le k \le \lfloor d/2 \rfloor$, the theorem follows by taking $j=\lfloor
d/2 \rfloor$ in Proposition~\ref{prop:neighborly}.  The remaining cases, i.e.\
$\lceil d/2 \rceil \le k \le d-2$, are obtained by taking the polar. 
\end{proof}

\section{Stamps}\label{sec:stamp}
Polytopes that are not $(0,d-3)$-scribable can be obtained by taking the polar of cyclic polytopes.  Here we present another 
alternative construction
based on projectively prescribed faces.

\begin{lemma}\label{lem:polygon_ellipse}
	For every $d\geq 2$, there is a polytope $\cP$ with no $(0,d-1)$-scribed
	projectively equivalent realization.
\end{lemma}

\begin{figure}[hptb]
	\includegraphics[width=.6\linewidth]{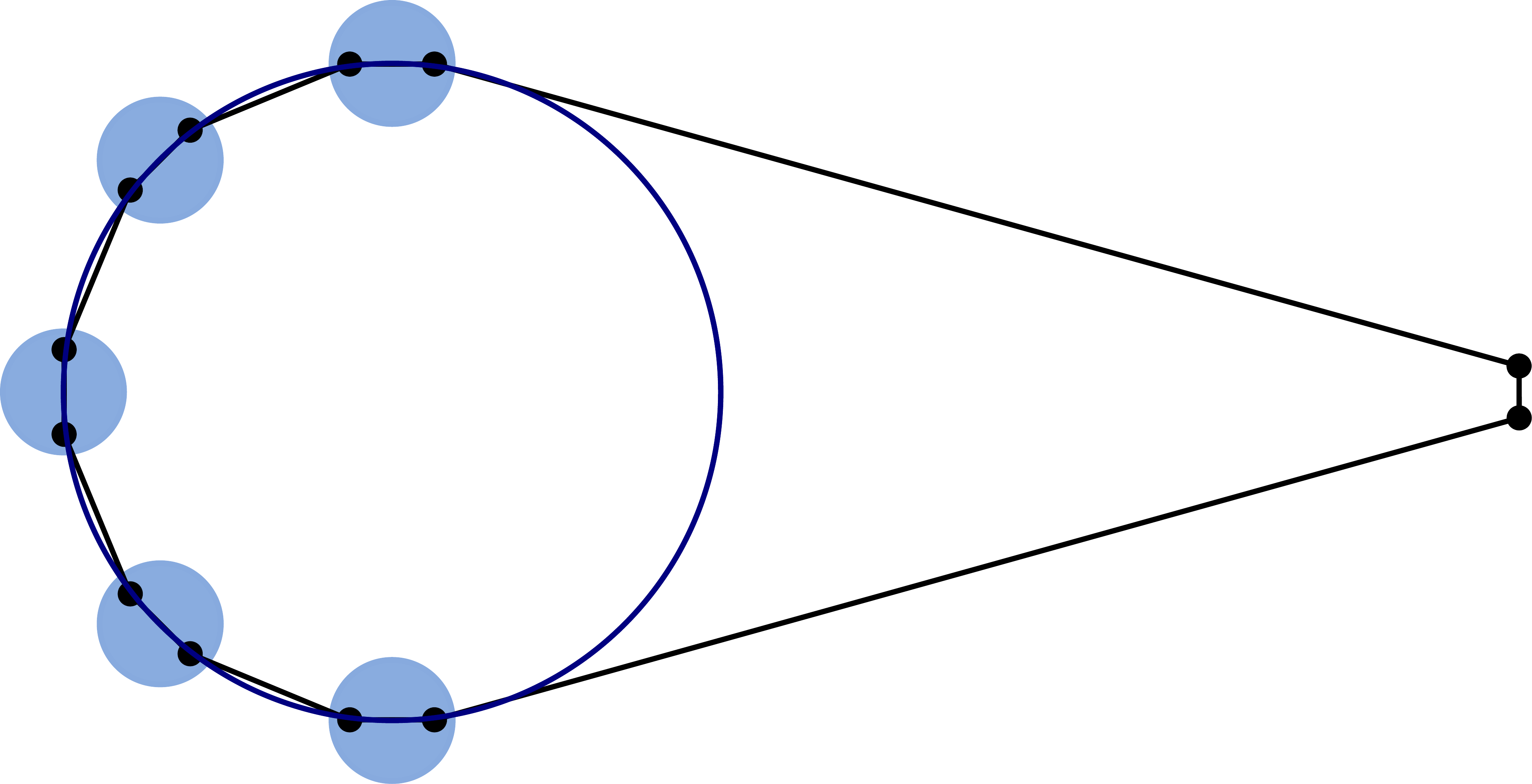}
	\caption{
		The construction of Lemma~\ref{lem:polygon_ellipse}. Any polygon
		projectively equivalent to this polygon is not $(0,1)$-scribed.
	}\label{fig:polygon_ellipse}
\end{figure}

\begin{proof}
	Consider $N=\binom{n+2}{2}-1$ generic points $p_1,\dots,p_N$ lying on the
	$x_0<0$ hemisphere of $\theSphere$.  There is a unique quadric going through
	$\binom{n+2}{2}-1$ generic points, and this dependence is continuous since
	the coefficients of the quadric are the solution of a linear system of
	equations on the points' coordinates.  Hence, there exists an $\varepsilon>0$
	such that for any $q_1,\dots,q_N$ with $q_i \in B_\varepsilon(p_i)$, the
	unique quadric that goes through these $q_i$'s is an ellipsoid contained in
	$2\theBall$.

	Now, consider $dN$ distinct points $p_{i}^j$ for $1\leq i\leq N$ and $1\leq j
	\leq d$ with $p_{i}^j \in \theSphere \cap B_\varepsilon(p_i)$.  Choose $d$
	additional points $p_0^j$, $1 \le j \le d$, on the hyperplane $x_0=3$ in the
	neighborhood $B_\varepsilon(p_0)$ where $p_0=(3,0,\dots,0)$, such that all
	the $d(N+1)$ points are in convex position.  Let $\cP$ be the convex hull of
	all these points.  If $\varepsilon$ is small enough, then for each $0 \le i
	\le N$, the corresponding $p_i^j$'s form a facet $F_i$ of $\cP$.

	For the sake of contradiction, assume a projective transformation $T$ such
	that $T\cP$ is $(0,d-1)$-scribed, then $T^{-1}\theSphere$ is a quadric that
	intersects all the facets of $\cP$.  Since $T^{-1}\theSphere$ contains a
	point $q_i \in B_\varepsilon(p_i)$ for each $1\leq i\geq N$, the quadric is
	contained in $2\theBall$ and hence does not intersect $F_0$.  Hence, such a
	transformation $T$ cannot exist. 

	This construction is sketched in Figure~\ref{fig:polygon_ellipse}.
\end{proof}

We need the following result, found by Below~\cite{below2002} and by Dobbins
\cite{dobbins2011,dobbins2013}.

\begin{proposition}[{\cite[Ch.~5]{below2002}}, see also {\cite[Thm.~4.1]{dobbins2011}} and {\cite[Thm.~1]{dobbins2013}}]\label{prop:Below}
	Let $\cP$ be a $d$-dimensional polytope with algebraic vertex coordinates.
	Then there is a polytope $\widehat{\cP}$ of dimension $d+2$ that contains a
	face $F$ that is projectively equivalent to $\cP$ in every realization of
	$\widehat{\cP}$.
\end{proposition}

 Such a polytope $\widehat{\cP}$ is called a \defn{stamp} for $\cP$ in
 \cite{dobbins2011,dobbins2013}. We are now ready to prove the main result of
 this part.
\begin{proposition}\label{prop:0d-3}
    A stamp for a $(d-2)$-polytope with no $(0,d-3)$-scribed projectively
	equivalent realization is not $(0,d-3)$-scribable.

\end{proposition}

\begin{proof}
	Let $\cP$ be a $(d-2)$-polytope with no $(0,d-3)$-scribed projectively
	equivalent realization, whose existence is guaranteed by
	Lemma~\ref{lem:polygon_ellipse}.  Observe that in the construction of
	Lemma~\ref{lem:polygon_ellipse} we can impose that it has algebraic
	coordinates (and even rational).  Now let $\widehat \cP$ be the stamp polytope from
	Theorem~\ref{prop:Below}.  We claim that $\widehat \cP$ is not
	$(0,d-3)$-scribable.  Otherwise, its $(d-2)$-dimensional face $F$, which is
	projectively equivalent to $\cP$ in every realization of $\widehat \cP$, is
	also $(0,d-3)$-scribed, contradicting our assumption.
\end{proof}

\section{Open problems}\label{sec:open}
Several natural question arise from our results. The most intriguing
is probably the existence of $d$-polytopes that cannot be $(0,d-1)$-scribed.

\begin{conjecture}
	For $d>3$, there are $d$-polytopes that are not $(0,d-1)$-scribable.
\end{conjecture}

Although we strongly believe that the conjecture is true, we did not manage to construct 
examples.  A promising strategy to find
such a polytope (suggested by Karim A. Adiprasito) would be using projectively unique polytopes, or polytopes with
a very constrained realization space. So far, the largest family of projectively unique
polytopes that we know of are those constructed by Adiprasito and
Ziegler~\cite[Section~A.5.2]{adiprasito2015}. However, they are essentially inscribable,
and hence they do not provide counterexamples directly.

Even for the case of $(1,d-1)$-scribability, our results are not complete.  We
only managed to find polytopes that are not $(1,d-1)$-scribable in even
dimensions: for odd dimensions $d \ge 5$, cyclic polytopes do not provide
examples; see Proposition~\ref{prop:odd}.  
\begin{conjecture}
	For every odd $d \ge 3$, there are $d$-polytopes that are not $(1,d-1)$-scribable.
\end{conjecture}

For odd dimensional cyclic polytopes, we know that they are $(1,d-1)$-scribable
(Proposition~\ref{prop:odd}) but not $(1,\floor{d/2})$-scribable
(Proposition~\ref{prop:neighborly}).  We would like to know

\begin{question}
	For odd $d \ge 5$ and $\floor{d/2} < k < d-1$, is every cyclic $d$-polytope
	$(1,k)$-scribable?
\end{question}

We showed that cyclic polytopes with sufficiently many vertices are not
circumscribable. We
conjecture that this holds for any neighborly polytope.

\begin{conjecture}\label{conj:neighnocircum}
	Neighborly polytopes with sufficiently many vertices are not
	circumscribable.
\end{conjecture}

If the conjecture is true, the dual of cyclic polytopes would give the first
examples of $f$-vectors that are not inscribable (see also~\cite{gonska2013}),
completing our Theorem~\ref{thm:fvector}.

On the other hand, every cyclic polytope is inscribable, and so are all neighborly
polytopes of a large family~\cite{gonskapadrol2016}. Computational results of
Moritz Firsching show that every simplicial neighborly $d$-polytope with at most $n$ 
vertices is inscribable when $(d,n)$ is $(4,\leq 11)$, $(5,\leq 10)$, $(6,\leq 11)$ or 
$(7,\leq 11)$~\cite{firsching2015}. 
In \cite{gonskapadrol2016} the following question is posed.

\begin{question}%
 Is every neighborly polytope inscribable?
\end{question}

For stacked polytopes, we do not have results on $(i,j)$-scribability for $j-i
\ge 2$.  Unfortunately, the angle-sum technique that proves
Theorem~\ref{thm:stack01} does not work for dimension $5$ or higher.

\begin{question}
	Given $i,j$ such that $j-i \ge 2$, is every stacked $d$-polytope
	$(i,j)$-scribable?
\end{question}

\bibliography{References}

\end{document}